\documentclass[a4paper, 12pt]{amsart}
\usepackage[foot]{amsaddr}
 \author{Jonas Hirsch}
 \email{hirsch@math.uni-leipzig.de}
 \address{Jonas Hirsch: Universit\"at Leipzig, Fakult\"at f\"ur Mathematik und Informatik, Augustusplatz 10, 04109 Leipzig}
 \author{Rob Kusner}
 \email{kusner@umass.edu, profkusner@gmail.com}
 \address{Rob Kusner: Department of Mathematics \& Statistics, University of Massachusetts, Amherst, MA 01003}
 \author{Elena M\"ader-Baumdicker}
 \email{maeder-baumdicker@mathematik.tu-darmstadt.de}
 \address{Elena M\"ader-Baumdicker (corresponding author): Technische Universit\"at Darmstadt, Fachbereich Mathematik, Schlossgartenstr. 7, 64289 Darmstadt, Germany}

\usepackage[english]{babel}
\usepackage{amsfonts,amsmath, amssymb}
\usepackage{amsthm}
\usepackage{xcolor}
\usepackage[left=3.5cm,right=3.5cm, bottom=4cm]{geometry}

\newenvironment{remark}[1][Remark]{\begin{trivlist}
\item[\hskip \labelsep {\bfseries #1}]}{\end{trivlist}}

\newtheorem{theorem}{Theorem}[section]
\newtheorem{lemma}[theorem]{Lemma}
\newtheorem{proposition}[theorem]{Proposition}
\newtheorem{corollary}[theorem]{Corollary}
\newtheorem{definition}[theorem]{Definition}

\newcommand{\C}{\mathbb C}
\newcommand{\Chat}{\hat{\mathbb C}}
\newcommand{\R}{\mathbb{R}}

\newcommand{\N}{\mathbb{N}}
\newcommand{\Sph}{\mathcal{S}^1}
\newcommand{\W}{\mathcal{W}}
\newcommand{\Szw}{\mathcal{S}^2}

\newcommand{\abs}[1]{\lvert#1\rvert}

\newcommand{\Sp}{\mathcal{S}}
\newcommand{\Ind}{\operatorname{Ind}_{\mathcal{W}}}

\newcommand{\Span}{\operatorname{span}}

\newcommand{\n}{\nu}

\begin{document}

\title[Willmore index]{Geometry of complete minimal surfaces at infinity and the Willmore index of their inversions }

% \author{Jonas Hirsch\thanks{Universit\"at Leipzig, Fakult\"at f\"ur Mathematik und Informatik, Augustusplatz 10, 04109 Leipzig, hirsch@math.uni-leipzig.de}, Rob Kusner\thanks{Department of Mathematics \& Statistics, University of Massachusetts, Amherst, MA 01003, kusner@umass.edu, profkusner@gmail.com}, Elena M\"ader-Baumdicker\thanks{(corresponding author) Technische Universit\"at Darmstadt, Fachbereich Mathematik, Schlossgartenstr.\ 7, 64289 Darmstadt, maeder-baumdicker@mathematik.tu-darmstadt.de}
% \date{\vspace{-4ex}}}
% 

\begin{abstract}
We study complete minimal surfaces in $\R^n$ with finite total curvature and embedded planar ends. After conformal compactification via inversion, these yield examples of surfaces stationary for the Willmore bending energy $\W: =\frac{1}{4} \int|\vec H|^2$.  In codimension one, we prove that the $\W$-Morse index for any inverted minimal sphere or real projective plane with $m$ such ends is exactly $m-3=\frac{\W}{4\pi}-3$. We also consider several geometric properties --- for example, the property that all $m$ asymptotic planes meet at a single point --- of these minimal surfaces and explore their relation to the $\W$-Morse index of their inverted surfaces. 
\end{abstract}

\maketitle

\section{Introduction}
We explore the connection between geometric properties of complete minimal surfaces with embedded planar ends in $\R^n$ and the Willmore Morse index of their inverted surfaces. The Morse index is the dimension of the maximal subspace of variations that locally decrease an energy to second order: in our case, the Willmore bending energy. 
Crucial to this connection is the M\"obius invariance of the Willmore energy. \\

We begin with a short review of Willmore surfaces --- that is, surfaces stationary for the Willmore energy --- and their relation to complete minimal surfaces with finite total curvature.
Given an immersed, closed surface $f:\Sigma\to\R^n$, we define the Willmore bending energy as 
\begin{align*}
 \W(f): =\frac{1}{4} \int_\Sigma |\vec H|^2 d\mu,
\end{align*}
where the mean curvature vector $\vec H$  is the trace of the second fundamental form of the immersion $f$. The symmetry group preserving  $\W$ consists of the M\"obius group of all conformal diffeo\-mor\-phisms of the ambient space $\R^n\cup\infty=\Sp^n$, the group generated by inversions in spheres (see, for example \cite{Blaschke, Thomsen}). A quantitative version of this invariance was pointed out by the second author \cite{KusnerII}: composing $f$ with the inversion $\varphi_x(y) = x + R^2\frac{y-x}{|y-x|^2}$ yields
\begin{align}\label{eq:confinv}
 \W(\varphi_x\circ f)= \W(f) -4\pi \theta(x),
\end{align}
where $\theta(x)= \sharp \{f^{-1}(x)\}$ is the number of preimages of $x$. Equality (\ref{eq:confinv}) has several consequences. The most important one for us is the following: consider a complete minimal surface with finite total curvature in $\R^n$; then inversion about a point $x\in\R^n$ off the surface yields a Willmore surface with a possible singularity at $x$.  In fact, as proven by Bryant \cite{BryantDuality} in $\R^3$, the resulting Willmore surface can be extended to a smooth immersion across the singularity if and only if all ends of the minimal surface are embedded and planar -- meaning each end is asymptotic to a multiplicity-$1$ plane at infinity. The other possibility for an embedded end is that it has logarithmic growth at infinity \cite{Schoen}. In this case, the corresponding inverted surface is only of regularity $C^{1,\alpha}$ for every $\alpha<1$, but it is not $C^{1,1}$ (see \cite{KusnerThesis88, Kuwert2004}). More literature about the smoothness and extendability of the Willmore equation can be found in \cite{Kuwert2008, Riviere2008, Riviere2013}.\\ 

The topology of $\Sigma$ is important for the problem of classifying Willmore surfaces. Bryant proved \cite{BryantDuality} the remarkable fact that any compact genus zero Willmore surface in $\R^3$ arises by compactifying an inversion of a complete minimal surface with embedded planar ends. This is of course not true for higher genus surfaces: a counterexample is the Clifford torus, a minimal surface in $\Sp^3$ whose stereographic projection to $\R^3$ is a Willmore surface. It is in fact the minimizer of the $\W$ functional among tori in $\R^3$ -- subject of the long-standing \emph{Willmore conjecture} -- proven only recently by Marques and Neves \cite{MarquesNevesWillmore, MarquesNeves}. Since the Clifford torus is embedded, an inverted Clifford torus can have at most one end. It cannot be an inversion of a complete minimal torus with embedded planar ends in $\R^3$. (In fact, up to homothety, the end of any embedded surface after inversion at a non-umbilic point is asymptotic to the graph of the biharmonic function $\cos 2\theta$ over the punctured plane in polar coordinates; this was the key observation in the first proof of existence for $\W$-minimizers with higher genus \cite{Kusner94, Bauer}.)\\

The result of Bryant \cite{BryantDuality} implies that the bending energies of Willmore spheres in $\R^3$ are quantized: each such surface has $\W(f) = 4\pi m$, where $m \in \N$ is the number of embedded planar ends of the corresponding minimal surface. Not all numbers are allowed: there are no Willmore spheres for $m=2,3,5,7$ \cite{Bryant},
but there are Willmore spheres with $\W(f) =   4\pi m$ with $m=2k$ for all $k\in \N\setminus\{1\}$; there are also examples with $m=2k+1$ for $k\geq 4$ (see the recent work of Heller \cite{Heller} based on earlier work of Peng-Xiao \cite{Peng}). 
%Michelat conjectures in \cite{Michelat9} that there are also no Willmore spheres for $m\geq 9$, $m$ odd. %%not clear see the contradicting result of Heller
The number $m=1$ corresponds to a plane which is the inversion of the round sphere, and the next possible number is $m=4$.\\

In \cite[Section~5]{Bryant}, Bryant analyzes the family of Willmore spheres (up to M\"obius transformation of the ambient space) with $\W=16\pi$: it turns out that there is a four-parameter family of these, whose moduli space is studied in \cite{KusnerSchmitt} via the spinor representation.  Two examples among them were already known: the Rosenberg-Toubiana surface \cite{RosenbergToubiana}, and the second author's explicit parametrization \cite{Kusner} of the minimal surface which gives rise after inversion to a \emph{Morin surface} \cite{Morin78}. A Morin surface has a $4$-fold orientation-reversing symmetry, meaning that (after an appropriate rotation in $\R^3$) it can be rotated by $\frac{\pi}{2}$ around the $x_3$-axis obtaining the same surface with the opposite orientation, making it a good \emph{half-way model} for a \emph{sphere eversion}: the remarkable fact that a round sphere can be turned inside out without creasing or tearing. \\

We give a short introduction to this phenomenon and explain why a good half-way model is important for sphere eversions.  The existence of a sphere eversion was proven by Smale \cite{Smale}: the round sphere in $\R^3$ with a given orientation is regularly homotopic to the round sphere with the opposite orientation, where two immersions $f_0, f_1:\Sigma \to\R^n$ are regularly homotopic if there is a $C^1-$path of immersions $F:\Sigma\times [0,1] \to\R^n$ such that $F(\cdot,0)=f_0$,$F(\cdot, 1)=f_1$ which induces a homotopy of the tangent bundles.
Following Smale's discovery, many mathematicians contributed to this field \cite{FrancisMorin, Apery, Morin78b, Morin78c}. In particular, it was shown by Banchoff and Max that along any eversion the path of immersed spheres must develop a \emph{quadruple point} \cite{BanchoffMax}: there is an immersion where a point in the image has four preimages. It follows from an inequality of Li and Yau \cite{LiYau} that a closed surface with a quadruple point has $\W\ge16\pi$. Morin (based on suggestions of Froissart, \cite{Morin78}) found a surface with one quadruple point and the orientation-reversing four-fold symmetry. This symmetry is very important for constructing a sphere eversion, since a deformation from that surface to a round sphere starting in the direction $v \nu$ ($\nu$ is a unit normal along the closed surface) joined with the opposite deformation along $-v \nu$ will automatically give a sphere eversion: a surface which does this job is the sought-for half-way model! \\

In the early 1980s the second author had the idea (see \cite{Kusner16}) of a sphere eversion such that at each stage of the deformation the immersions have the least possible Willmore energy: an \emph{optimal sphere eversion}. Such an eversion needs a suitable half-way model with $\W=16\pi$ energy, perturbing that surface so that the Willmore energy decreases and then starts the Willmore flow, that is, the gradient flow of $\W$. So with a discretization of the second author's $\W$-stationary Morin surface half-way model as the starting place, using the Brakke {\em Evolver} the authors of \cite{Francis95} %Francis, Sullivan, the second author, Brakke, Hartman and Chappell 
were able to compute an optimal (discretized) animation of a sphere eversion similar to one Morin and Petit had described 
\cite{Morin78}. 
To start (half of) the optimal sphere eversion from the Morin surface half-way model, one first needs to perturb the Morin surface such that the Willmore energy decreases. The authors of \cite{Francis95, Francis97} computed numerically that there is a $1$-dimensional space of variations of the Morin surface decreasing the Willmore energy $\W$ to second order \cite[Section~3]{Francis95}: that is, the Morin surface has \emph{Morse index} $1$.\\

It remains an open problem whether, starting from the (suitably perturbed) Morin surface with $\W=16\pi$, the Willmore flow will lead to the round sphere, and thus produce a \emph{smooth} optimal sphere eversion: this is sometimes called \emph{the $16\pi$ conjecture}.
%[-0.2cm]
 In 2015, Rivi\`{e}re \cite{Riviere2018} developed an alternative approach to tackle the $16\pi$ conjecture. He proved, among other things, that a positive min-max width for the spherical topology is realized by a bubble-tree of (possibly branched) Willmore spheres. In the following year, Michelat \cite{Michelat} computed the formula for the second variation of the Willmore functional for inverted minimal surfaces with $m$ embedded planar ends which he used to show the upper bound
\begin{align*}
  \Ind(f) \leq m.
\end{align*}
In 2019, an improved upper bound and a lower bound was proven
\cite{HirschEMB} by the first and third author, who established the following formula for the $\W$-index of a Willmore sphere in $\R^3$:
\begin{align}
 \Ind(f) = m -d,
\end{align}
where $m$ is the number of planar ends of the corresponding minimal surface $X:\Szw\setminus\{p_1,...,p_m\}\to\R^3$, and $d:=\dim \operatorname{span}\{\n(p_1),...,\n(p_m)\}$ is the dimension of the span of the asymptotic normal vectors to $X$ at the ends; they also show $d=3$ for $m=4$, and $2\leq m\leq 3$ for general $m$. At the same time and independently, Michelat \cite{MichelatBranched} was able to generalize an upper bound to branched Willmore spheres appearing as inverted minimal surfaces. For example, he showed $\Ind(f) \leq m-1$ for a branched Willmore sphere arising as a compactified complete minimal sphere with finite total curvature and zero flux. This follows from an approach identifying the $\W$-Morse index as the index of an $m\times m$ matrix related to certain bi-Jacobi fields. In both \cite{HirschEMB} and \cite{MichelatBranched} it is crucial to allow for variations on the complete minimal surface containing a term with logarithmic growth.\\

One of the main results of this article is to prove $d=3$ in general for complete minimal spheres (and real projective planes). Our first result is the following, see Section~\ref{Section2}: 

\begin{theorem}
Let $X: \Sigma\setminus\{p_1,...,p_m\} \to\R^3$ be a complete
conformal minimal immersion with embedded planar ends at $\{p_1,...,p_m\}\subset\Sigma$ with $m>1$. If $\Sigma$ is a sphere or a real projective plane, then we denote by $\n:\Sp^2 \to \Sp^2\subset\R^3$ the holomorphic extension of the Gau\ss\ map of $X$ or its orientable double cover.
%minimal surface with embedded planar ends that is conformally parametrized over $\Sigma\setminus\{p_1,...,p_m\}$ where $p_1,...,p_m \in \Sigma$, $m>1$, correspond to the ends of the surface.
 %Let $\n:\Sp^2 \to \Sp^2\subset\R^3$ be the holomorphic extension of the Gau\ss\ map of $X$. 
 Then the asymptotic normals satisfy
 \begin{align}
  d: = \dim \Span \{\n(p_1),..., \n(p_m)\} = 3.
 \end{align}
 As a consequence, the Morse index of the inverted surface $f:\Sigma \to\R^3$ is
 \begin{align}
   \Ind(f) = \frac{1}{4\pi} \W(f) -3.
 \end{align}
\end{theorem}
In \cite{HirschEMB}, the authors were able to draw conclusions for the Morse index of a Willmore sphere by studying area-Jacobi fields -- functions in the kernel of the second order elliptic operator corresponding to the second variation of the area functional. Certain questions related to the geometry of complete minimal surfaces with embedded planar ends arose that are of independent interest. One such question is whether the asymptotic planes of the complete minimal surface always meet at a single point in $\R^3$. We answer this question partly in Section~\ref{Section3}. In \cite{Kusner}, the second author constructs a family of symmetric minimal surfaces with $m=2p$ embedded planar ends. We call these surfaces \emph{minimal flowers}. Using the spinor representation of minimal surfaces \cite{KusnerSchmitt} we prove the following:
\begin{theorem}
 Let $X:\Sp^2\setminus\{p_1,...,p_m\} \to\R^3$ be a complete minimal sphere with embedded planar ends that lies in the $\Sph\times SO(3,\C)$-orbit (see Section~\ref{Section3}) of a minimal flower. Then all the asymptotic planes of the ends meet at one point.
\end{theorem}
Note that, in the case of four-ended surfaces, the whole moduli space of complete minimal spheres with embedded planar ends consists of the $\Sph\times SO(3,\C)$-orbit of the inverted Morin surface \cite{KusnerSchmitt}.
With this property in hand, we can simplify the proof \cite{HirschEMB} that each Willmore sphere with $\W=16\pi$ has $\W$-Morse index one. In Section~\ref{Section3}, we also give a weaker condition than in \cite{HirschEMB} for a complete minimal surface with embedded planar ends of arbitrary genus to have $\W$-Morse index at least $1$ after inversion.\\[-0.2cm]

Another consequence of the computation for the second variation of $\W$ is a statement about what we call the \emph{conformal density at infinity}. The classical density at infinity of a complete minimal surface is 
\begin{align}
 \lim_{r \to \infty} \frac{|\mathbb{B}_{r}\cap X(\Sigma)|}{\pi r^2},
\end{align}
which is $m$ for a minimal surface with $m$ embedded planar ends due to the monotonicity formula, see for example \cite{HoffmanMeeks}. We suggest considering a different, more intrinsic \emph{conformal density at infinity}. We use conformal coordinates with~$z(p_i)=0$ and $X_z dz = \left(-\frac{a}{z^2} + Y(z)\right) dz$ with $a\in \C^3, a\cdot a=0, |a|^2=2$ and $Y$ is holomorphic at each end $p_i$ and define $D_\epsilon(p_i) := z^{-1}(B_\epsilon^{\C})$. Then our computations in Section~\ref{Section4} show that the convergence of the related density quantity is stronger.
\begin{proposition}
 Let $\Sigma$ be a closed surface and 
 $X:\Sigma\setminus\{p_1,...,p_m\} \to\R^3$ a complete minimal immersion with $m$ embedded planar ends. With the choice of coordinates made above, we have
 \begin{align*}
  \lim_{\epsilon\to 0} \left\{\int_{\Sigma\setminus \bigcup_{i=1}^m D_\epsilon(p_i)}d\mu_g - \frac{\pi m}{\epsilon^2}\right\} =0.
 \end{align*}

\end{proposition}
Finally, in Section~\ref{Section5}, we study the $p$-equivariant $\W$-Morse index --- that is, 
the maximal dimension of the subspace of $p$-fold rotation-symmetric variations that decrease $\W$ to second order --- of the minimal flowers with $2p$ ends described in \cite{Kusner}. In \cite{HirschEMB}, we already showed that the Morin surface (which is, after inversion, a minimal flower) has $\W$-Morse index 1, and, in fact, the variational direction that decreases $\W$ to second order most rapidly must be $2$-fold rotation symmetric. Here, we show that the $p$-equivariant Willmore Morse index of an inverted minimal flower is always 1 as well:
\begin{theorem}
Let $f:\Sp^2 \to\R^3$  be a closed, immersed Willmore sphere such that $X:= \frac{f}{|f|^2}:\Sp^2\setminus\{p_1,..., p_{2p}\} \to \R^n, p \in \N$, is a complete, immersed minimal sphere with $2p$ embedded planar ends. Assume further that $f$ has an orientation reversing $2p$-fold rotational symmetry around an axis of symmetry going through $0=f(p_i)=f(p_j)$  for all $i,j\in\{1,...,2p\}$. 

 Under these conditions, the subspace of $p$-fold rotation-symmetric variations of $f$ decreasing $\W$ to second order is $1$-dimensional.
 (If $p$ is prime, all other variations corresponding to negative eigenvalues of the $\W$-Jacobi operator break the symmetry.) 
 \end{theorem}
We also show in Section~\ref{Section5} the expected property that the eigenfunction of the lowest eigenvalue of the $\W$-Jacobi operator on the Morin surface must have a sign change.\\[-0.2cm]

\subsection*{Acknowledgments}

We would like to thank Karsten Gro\ss e-Brauckmann for several illuminating discussions throughout the project, and we 
 appreciate Alexis Michelat pointing out several recent references.
%are also grateful to  Alexis Michelat for pointing us to several recent  references. 
The first author was partially supported by the German Research Foundation in context of the Priority Program SPP 2026 ``Geometry at Infinity''.
The second author is grateful to CIRM in Luminy and FIM at ETH Z\"urich for hosting his visit to Europe in May and June 2019 when this collaboration commenced, and he also enjoyed the support of KIMS and Coronavirus University during the pandemic. 
The third author is financially supported by the German Research Foundation (MA 7559/1-1 and MA 7559/1-2) and she is grateful for the support.

\subsection*{MSC classification}
53C42, 53A10, 35J35, 35R01

\section{Dimension of the span of the normals at infinity} \label{Section2}

 In this section we study the tangent spaces at infinity for a complete minimal surface with finite total curvature and embedded planar ends. 
We begin with an observation (perhaps known to other experts) about the dimension into which such a surface can be fully immersed:
\begin{lemma}\label{lem.number of ends 1}
Let $X: \Sigma\setminus \{p_1,\dotsc, p_m\} \to \R^n$ be a complete minimal surface with finite total curvature and embedded planar ends at $p_i$, $i=1,\dotsc, m$. If $m <\frac{n}{2}$, then $X$ is contained in some affine $\R^{n-1}$. 	
\end{lemma}
\begin{proof}
	Let $\pi_i$ be the orthogonal projection onto the tangent space $T_{p_i}:=T_{p_i}\Sigma$ of $X(\Sigma)$ at the end $p_i$,  regarded as a  {\em linear subspace} of $\R^n$. Consider the linear map $A:=(\pi_1, \dotsc, \pi_m)$ that maps $\R^n$ onto $T_{p_1} \times \cdots \times T_{p_m}\cong\R^{2m}$. If $2m<n$, we have $\ker(A) \neq \{0\}$. Pick a nonzero $v \in \ker(A)$ and consider the height function $u:=v\cdot X $. 
	Since  $u$ is harmonic on $X(\Sigma\setminus \{p_1,\dotsc, p_m\})$ and $v \in \pi_i^\perp$ for all $i$, we conclude that $u$ is bounded and therefore constant. 
\end{proof}

\begin{remark} For any minimal surface with finite total curvature and embedded planar ends at $\{p_1,...,p_m\}$, a similar argument to the proof of Lemma~\ref{lem.number of ends 1} shows that the tangent planes at infinity coincide $\bigcap_j T_{p_j} = T_{p_1}$ if and only if the minimal surface is a plane \cite{HirschEMB}.
\end{remark}

\begin{remark}
Aside from a union of planes, there is no complete minimal surface in $\R^3$ with finite total curvature and $m=2$ embedded planar ends \cite{Schoen, BryantDuality, Kusner}, so Lemma \ref{lem.number of ends 1} implies such a surface lies fully in $\R^4$ and thus the intersection of its tangent spaces at infinity is the zero space.
\end{remark}

The lemma below shows this remains true generically for genus 0 minimal surfaces with at least 2 embedded planar ends:

\begin{lemma}\label{lem:genericS}
Let $X\colon\Sp^2\setminus \{p_1, \ldots , p_m\} \to \R^n$ be a complete minimal sphere with finite total curvature and embedded planar ends at $\{p_1, \ldots , p_m\}$ and tangent space $T_{p_j}$ at $p_j$. Assume $X$ is not a plane, so $m\ge2$.
Then generically
\begin{equation}\label{eq.intersection of tangent planes at infinity} \bigcap_{1\le j\le m} T_{p_j}=\{0\} \,.\end{equation} 
%Furthermore, if the set of planes $\{T_{p_1}\Sigma,...,T_{p_m}\Sigma\}$ consists only of two different planes $L_1, L_2$, then $L_1\cap L_2=\{0\}$.
%there are only two different ends in the sense that there is $L_1, L_2$ such that $T_{p_j}\Sigma\in \{L_1,L_2\}$ then $L_1\cap L_2=\{0\}$. 
\end{lemma}

\begin{proof}
Because our minimal surface has genus 0 and planar ends, differentiating its conformal parametrization $X$ with respect to a local complex coordinate $z$ on the domain surface gives a nonvanishing $\C^n$-valued meromorphic differential 
$\partial X = X_z dz$ 
with zero --- both real and imaginary --- periods.  The conformality  is expressed by $(X_z)^2=X_z\cdot X_z=0$, which is preserved --- as is the zero periods condition --- under matrix multiplication by any $M\in SO(n,\C)$.  Thus we can integrate the differential $MX_z dz$ and take its real part to get a new minimal surface with genus 0 and planar ends parametrized conformally by $X_M: = 2\Re{ \int^z MX_z  \, dz}$ (unique up to translation) from the same domain surface.  

In this way the Lie group $SO(n,\C)$ acts on the space of genus 0 minimal surfaces in $\R^n$ with embedded planar ends and finite total curvature.  
The real subgroup $SO(n,\R)$ acts in the usual way by rotation on $\Sp^{n-1}\subset\R^n$.  The action of the entire group $SO(n,\C)$ is harder to describe geometrically, but in case $n=3$ the $SO(3,\C)$ action on $\Sp^2\cong\C\cup\infty$ is by M\"obius (or fractional linear) transformations using the isomorphism $SO(3,\C)\cong PSL(2,\C)$ (see \cite{Kusner, KusnerSchmitt}); for instance, observe that the 1-parameter subgroups along $i\mathfrak{so}(3,\R)$ deform $\Sp^2$ --- and by post-composition, deform the Gau\ss\ unit normal map of the minimal surface in $\R^3$ --- by centered dilations.

We can now use this observation (parts of which were described in \cite{KusnerSchmitt}) to give a simple geometric argument for the case $n=3$ as follows:

\underline{Case 0}: $\bigcap_j T_{p_j}=T_{p_{j_0}}$ is equivalent to $\dim(\Span\{\nu(p_{1}),...,\nu(p_{m})\}) =1$, yielding a bounded harmonic function, implying the surface is a plane as in the Remark above. \\ 
So for a minimal sphere with  $\dim (\Span\{\nu(p_{1}),...,\nu(p_{m})\}) =2$ at its $m\ge3$ planar ends, we need to find a perturbation via planar-ended minimal surfaces satisfying $\dim(\Span\{\nu(p_{1}),...,\nu(p_{m})\}) =3$ which is equivalent to (\ref{eq.intersection of tangent planes at infinity}).

\medskip
\underline{Case 1}: $\dim(\Span\{\nu(p_{1}),...,\nu(p_{m})\}) =2$ and the asymptotic normals at the ends hit at least three different points on $\mathcal{S}^2$. \\
In this case, the (at least) three values of the Gau\ss\ map lie on a great circle of $\Sp^2$. Now we move them off that great circle via a centered dilation with source and sink at the poles of that great circle. This implies that the family of planar-ended minimal surfaces where the normals at the ends span $\mathbb{R}^3$ is open and dense, i.e. generic. 

\medskip
\underline{Case 2}: $\dim(\Span\{\nu(p_{1}),...,\nu(p_{m})\}) =2$ and the asymptotic normals at the ends hit only two different points on $\mathcal{S}^2$.\\
This is impossible: if there were such a minimal surface with embedded planar ends, we could move the end normals by a M\"obius transformation to a pair of antipodal points, giving a height function along that direction which is bounded, implying the surface is a plane. 
%\end{proof}

%how the $SO(n,\mathbb{C})$-action 

%can be used to perturb minimal surfaces in $\mathbb{R}^n$.
\medskip

Next we recall the complex null quadric description \cite{HoffmanOsserman, Lawson} for the Grassmannian of (oriented) real 2-planes in $\mathbb{R}^n$, which is the codomain for the Gau\ss\ map in the higher codimension  case.
Consider the projective algebraic subvariety $\mathcal{Q}\subset \C P^{n-1}$ given by 
\[\mathcal{Q}=\{ [Z] \in \C P^{n-1}: Z\in \C^n\setminus\{0\} \text{ with } Z^2 = Z\cdot Z=0\}\]
where the polynomial $Z\cdot W= \sum_{i=1}^n Z_iW_i$ considered in homogeneous coordinates is the complexification of the real inner product on $\C^n$.   Hence to each element $[Z]\in \mathcal{Q}$ we associate the oriented real 2-plane $L_{Z}=\operatorname{span}\{\Re{Z}, \Im{Z}\}$.
Given a complete minimal surface $X\colon \Sigma\setminus\{p_1,\dotsc, p_m\} \to \R^n$ with finite total curvature and embedded planar ends $p_j, 1\leq j \leq m$, the associated meromorphic differential 
\[\partial X= X_z \, dz\]
satisfies $(X_z)^2=0$ and it defines a complex line in $\C^n$ that is independent of the choice of coordinates \cite[page 110]{Lawson}. Thus, at each $p\in\Sigma\setminus\{p_1,...,p_m\}$ there is an element in $\mathcal Q$ representing the differential. Since all ends are planar, also at $p_j$ there is a well-defined complex line associated to $p_j$. Geometrically the element in $\mathcal{Q}$ corresponds to the oriented tangent space of the minimal surface in $\R^n,$ defining its Gau\ss\ map to the Grassmannian of oriented real 2-planes in $\R^n$. 
Conversely, given a nonvanishing meromorphic differential $\phi$ on a punctured Riemann surface $\Sigma$ taking values in $\mathcal{Q}$ and which has vanishing real periods, then
\[ X(z) = 2\Re{ \int^z \phi \, dz}\] 
 conformally parametrizes a minimal surface with $X_z\,dz =\phi$. 

Observe that the action of $SO(n,\C)$ leaves the variety $\mathcal{Q}$ invariant and defines an action on the Gau\ss\ maps (analogous to the $n=3$ case above) and thus on the space of genus 0 minimal surfaces (or more generally, the space of minimal surfaces with zero periods) with embedded planar ends (see for example \cite{Lee2}, in case $n=4$). We will use this Lie group action to complete the proof of the lemma for the general case $n\geq 3$.

\medskip
Although we have $m$ planar ends, some of the tangent spaces $T_{p_j}$ at the points $p_1,...,p_m$ can coincide. We denote by $L_1,...,L_M$ the $M$ pairwise different tangent planes at the ends, i.e.\ we have a coincidence of sets \{$T_{p_1},..., T_{p_m}\}=\{ L_1,..., L_M \}$, where $M\leq m$.
\medskip

\emph{Observation 0:} If $M=1$, then $X(\Sigma)$ is a plane:
For every $v \in \R^n$ that is orthogonal to $L_1$, $z \mapsto v\cdot X(z)$ is a bounded harmonic function on $X$ and so constant. Hence we deduce that $X(\Sigma) \subset p+  L_1$. 

\medskip
 So $M\ge2$, and suppose to the contrary that  $\bigcap_{k=1}^M L_k = \bigcap_{j=1}^m T_{p_j}\not =\{0\}$. Following a rotation, we may assume that 
\[e_1\R \subset \bigcap_{1\le j\le m} T_{p_j}\,.\]

%\medskip

\emph{Claim 1:} $M=2$ is not possible. 

In this case, we have $\{L_1,L_2\}=\{L_{Z_1},L_{Z_2}\}$ and we will show that there exists an $A \in SO(n, \C)$ such that $A[Z_1]=[\bar{Z}_2]$. This implies that $A X_z\,dz$  is a non-vanishing meromorphic differential on $\Sigma$ taking values in $\mathcal{Q}$ with zero periods. Consequently, 
\[ X_A = 2 \Re \int^z AX_z\,dz \]
defines a minimal sphere as considered in Observation 0. 
This, however, implies that $X_A$ is a plane. Consequently, also $X$ itself is a single plane, which is contrary to the assumption that $M=2$. \\
\medskip
For the existence of $A$ we proceed as follows.
After a rotation we may assume that $z_j = e_1 + i (a e_2 + (-1)^j b e_3)$ with $a,b\in \R,\, a^2+b^2=1$. For every $t\in \R$ 
\[A(t) := \begin{pmatrix}
	\cosh(t)&-i\sinh(t)&\\i\sinh(t)&\cosh(t)&\\
	&&\mathbf{1}_{\C^{n-2\times n-2}}
\end{pmatrix}\]
is an element of $SO(n,\C)$ generated by $i\Lambda=i \left(e_2\otimes e_1 - e_1 \otimes e_2\right)\in i\mathfrak{so}(n,\R)$. If we select the value of $t$ such that $-a=\sinh(t)$, we obtain the desired element. 

\medskip
\emph{Claim 2:} If $M\ge 3$, then there exists an element $E(t)=\exp(t i \Lambda) \in SO(n,\C)$ such that $X_{E(t)}$ satisfies \eqref{eq.intersection of tangent planes at infinity} for all $t\neq 0$.

We assume that three planes $L_{Z_1}, L_{Z_2}, L_{Z_3}$ share the $e_1$-direction. After performing a further rotation we may assume that $Z_j= e_1 + i v_j, j=1,2, 3$, represent the three planes and $\operatorname{span}\{ v_1,v_2\} = \operatorname{span}\{e_2,e_3\}$.  Note that the element $i \Lambda = i (e_3\otimes e_2 - e_2 \otimes e_3)\in i\mathfrak{so}(n,\R)$ generates 
\[E(t)%:=\exp(t i \Lambda ) 
= \begin{pmatrix}
	1&&&\\&\cosh(t)&-i\sinh(t)&\\&i\sinh(t)&\cosh(t)&\\
	&&&\mathbf{1}_{\C^{n-3\times n-3}}\end{pmatrix}\,.\]
It remains to show that there is no common line in the three planes
\begin{align*}E(t)Z_j&= (e_1-\sinh(t)\Lambda v_j) + i \cosh(t) v_j \text{ for } j=1,2\, ,\\
E(t)Z_3 &= (e_1-\sinh(t)\Lambda v^\top_3) + i\left( \cosh(t) v^\top_3 + v_3^\perp\right) \end{align*}
where $w=(w^\top,w^\perp) \in \operatorname{span}\{e_2,e_3\}\times\operatorname{span}\{e_n\colon n >3 \}$.  

If $v_3^\perp \neq 0$ then $L_{E(t)Z_3} \cap  \operatorname{span}\{e_1,e_2,e_3\}= \operatorname{span}\{e_1-\sinh(t)\Lambda v^\top_3\}$ and this line is not contained in $L_{E(t)Z_j} \subset   \operatorname{span}\{e_1,e_2,e_3\}, j=1,2$. This can be seen as follows: the normal space to $L_{E(t)Z_j}, j=1,2$ in $\operatorname{span}\{e_1,e_2,e_3\}$ is spanned by $N_j = (\sinh(t), \Lambda v_j), j=1,2$ and  
\[ N_j\cdot (e_1-\sinh(t) \Lambda v_3^\top) = \sinh(t) \left( 1 - v_j\cdot v_3^\top\right)\neq 0 \]
because $|v_j\cdot v_3^\top|\le |v_3^\top|<1$.

It remains to consider the case $v_3\in \operatorname{span}\{e_2,e_3\}$. However, then the three planes $L_{E(t)v_j}$ contain a line if and only if the three normal vectors $N_j= (\sinh(t),\Lambda v_j)$ are linearly dependent. But this is not the case which can be seen as follows. Suppose by contradiction that for some $\beta\neq 0$ we have $\sum_j \beta_j N_j=0$. This is equivalent to $\sum_j \beta_j v_j =0$ and $\sum_j \beta_j =0$. After relabeling the vectors we may assume that $-\beta_1 < 0 \le  \beta_2\le \beta_3$. But then we have
\[0<-\beta_3= \beta_2|v_2|+\beta_3|v_3| \le |\beta_2v_2+\beta_3 v_3|= |\beta_3v_3|=-\beta_3\,.\]
Therefore, we must have equality everywhere, which implies that all vectors must be parallel. However, this is contrary to the assumption that we are considering three distinct planes.
\end{proof}

%\begin{remark}
   %          Here, we use essentially that $X$ has genus 0 so that we do not have to care about periods.  
%\end{remark}

We recall the definition of the index form and the Jacobi operator for the Willmore functional found in \cite[Section~2]{HirschEMB}.

\begin{definition}\label{def:index}
Let $f: \Sigma\to \R^n$ be a smooth, closed Willmore immersion and $g$ is the induced metric.  Let $\Gamma(N\Sigma)$ be the smooth sections of the normal bundle of $f:\Sigma\to\R^n$. Then there is a strongly elliptic, $L^2$-self-adjoint operator 
\begin{align*}
Z: \Gamma(N\Sigma) \to \Gamma(N\Sigma)
    \end{align*}
of fourth order, called the $\W$-\emph{Jacobi operator of }$\mathcal W$, such that 
\begin{align*}
 \delta^2 \W(f)(\vec v, \vec v): = \int_{\Sigma}  \vec v\cdot Z\vec v\, d\mu_{g},
\end{align*}
where $Z$ is characterized by the property
\begin{align}
 \frac{d^2}{d t^2} \W (f_t)\big|_{t=0} = \int_{\Sigma}  \vec v\cdot Z\vec v\, d\mu_{g},
\end{align}
for a \textbf{smooth} variation $f_t:\Sigma\to \R^n$, $t\in (-\epsilon, \epsilon)$, i.e.\ $f_t$ is a smooth immersion for $t\in(-\epsilon, \epsilon)$ and $f_0 = f$. Here, $\vec v :=\frac{d}{dt}f_t\big|_{t=0}$ is the variational vector field. The object $\delta^2 \W(f)(\vec v, \vec v)$ is called \emph{index form} and is well-defined for sections $\vec v$ that have regularity $W^{2,2}$ with respect to $g$ (see the remark below).\\
If $n=3$ and $\Sigma$ is oriented with global unit normal $\nu$, then we also use the symbol $Z$ for the $\W$-Jacobi operator acting on functions $v$ via $\vec v=v\nu$.\\
By the theory of strongly elliptic self-adjoint operators, $Z$ can be diagonalized on $\Gamma(N\Sigma)$ with eigenvalues
$$\lambda_0 \leq \lambda_1\leq ... \leq \lambda_k\leq ...$$
whose eigenspaces $Y_{\lambda_j}$ corresponding to $\lambda_j$ are finite dimensional.\\
The \emph{$\W$-Index} and the $\W$-nullity of $f$ are defined as follows:
\begin{align*}
 \Ind(f)& : = \dim \left(\oplus_{\lambda<0} Y_{\lambda}\right),\\
 \operatorname{nullity}_{\W}(f)& : = \dim \left(Y_0 \right).
\end{align*}
A Willmore surface $f$ is called $\W$-\emph{stable} if $\Ind(f) = 0$. This is equivalent to $\delta^2 \W(f)(\vec v, \vec v) \geq 0$ for all $\vec v \in \Gamma(N\Sigma)$.

  \end{definition}
  
 \begin{remark}  We note several features of the $\W$-Jacobi operator and index form:
 
 \begin{itemize}
 
 \item For any smooth variation $f_t$, we can always arrange that $\vec v =\frac{d}{dt}f_t\big|_{t=0}$ is normal by composing with a tangential diffeomorphism.

 \item That the operator $Z$ is formally $L^2$-self-adjoint was explicitly checked in \cite[p.\ 14]{Lamm1} for codimension $1$ and general ambient manifolds. We provide a general argument for self-adjointness in Appendix~\ref{appendix:selfadjoint}.
      
       \item  Let $L:\Gamma(N\Sigma) \to\Gamma(N\Sigma)$ be the Jacobi operator of the area functional in $\R^n$, i.e.\ $$L\vec v = \Delta^\perp \vec v + \tilde A (\vec v),$$ 
       where $\Delta^\perp$ is the Laplacian on the normal bundle and $\tilde A$ is the Simons operator $\tilde A(\vec v)= \sum_{i,j} g(A(e_i,e_j), \vec v)A(e_i,e_j)$ for an orthonormal frame $e_i$, see for example \cite[Section~8]{Colding}.  For $n=3$ and variations of the form $\vec v = v\nu$, the formula for $L$ reduces to $L(v) = \Delta_g v +|A|^2 v$.
       
       Then the $\W$-Jacobi operator $Z$ can be written in the form $Z \vec v = L L \vec v+ \Omega \vec v$, where $\Omega:\Gamma(N\Sigma) \to\Gamma(N\Sigma)$ is a self-adjoint operator of order two, see for example \cite[Section~3]{Lamm1}
    for codimension $1$. This form $Z = L^2 +\Omega$ follows in general from the formula for $\partial_t \left(|\vec H|^2d\mu_g\right)$ which can be found, for example, in \cite[Formula~(1.3.5) on p.~10]{KuwertSch}: 
    \begin{align*}
      \partial_t \left(|\vec H|^2d\mu_g\right) = 2\left(\langle \Delta^\perp V, \vec H\rangle + g^{ik}g^{jl}\langle A^\circ_{ij}, \vec H\rangle \langle A^\circ_{kl}, V\rangle\right)d\mu_g.
    \end{align*}
   This formula implies, after integration by parts, that the second variation at a critical point involves $\left(\Delta^\perp\right)^2$ as a leading term plus second order terms which is equivalent to the claimed structure. Note that this does not imply that $Z$ is free of third order terms: when expanding $\left(\Delta^\perp\right)^2$ third order terms may appear.
  
   \item Since $\left(\Delta^\perp\right)^2$ is the leading part of $Z$ and $\Sigma$ is compact, we see that $Z$ is strongly elliptic. 
     
    \item  The index form $\delta^2 \mathcal{W}$ is continuous with respect to $W^{2,2}$-convergence,  so for any $\vec{v} \in W^{2,2}(\Sigma, N\Sigma)$ there is a sequence of smooth sections $\vec{v}_k \in \Gamma(N\Sigma)$ with $\vec{v}_k \to \vec{v} \in W^{2,2}$ and 
   \[ \delta^2\mathcal{W}(f)(\vec{v},\vec{v})= \lim_k\delta^2\mathcal{W}(f)(\vec{v}_k,\vec{v}_k)  = \lim_{k\to\infty} \frac{d^2}{dt^2} \mathcal{W}(f + t\vec{v}_k)\big|_{t=0}\,.\]
   By the above construction we extend the variations of $f$ to sections that are only in $ W^{2,2} $.
                \end{itemize}
 \end{remark}

\begin{corollary}\label{thm:span at the ends in R3}
Let $X: \Sp^2\setminus\{p_1,...,p_m\} \to\R^3$ be a complete genus $0$ minimal surface with finite total curvature and embedded planar ends at $\{p_1,...,p_m\}$, $m>1$. If $\nu:\Sp^2 \to \Sp^2\subset\R^3$ denotes the holomorphically extended Gau\ss\ map of $X$, then
 \begin{align}
  d: = \dim (\Span \{\n(p_1),..., \n(p_m)\}) = 3.
 \end{align}	
\end{corollary}

\begin{proof}
Note again that $m>1$ implies $m\geq 4$ already. 
In the work \cite{HirschEMB} it was proven that the $\W$-index of an inverted complete minimal with $m$ embedded planar end is $m-d$ and that $d\in \{2,3\}$. 
If there were a complete minimal sphere $X$ with $m$ embedded planar ends and $d=2$, the $\W$-index of its inverted surface would be $m-2$. By Lemma~\ref{lem:genericS} we know that $\Ind = m-3$ on an open and dense set of $\W$-critical spheres with $\W=4\pi m$ satisfying $d=3$. Thus, the existence of such a $\W$-critical sphere $X$ would contradict the lower semi-continuity (Appendix \ref{appendix:lowersemicontinuity}) of the $\W$-index.	
\end{proof}

\begin{corollary}\label{cor.index estimate in higher codimension}
Let $X: \Sp^2\setminus\{p_1,...,p_m\} \to\R^{2+k}$, $2+k\le 2m$, be a complete minimal sphere with finite total curvature and embedded planar ends at $\{p_1,...,p_m\}$, $m\ge 3$. If $f$ denotes the inversion of $X$, whose compactification is a Willmore sphere, we have  
 \begin{align}
  \Ind(f) \le km - (2+k).
 \end{align}	
\end{corollary}
\begin{proof}
In \cite{HirschEMB} it was proven that $\Ind(f) \le km - d$ where $d=(2+k) - l$ and $l$ is the dimension of the intersection of the tangent  spaces at the ends. Since generically we have $l=0$ by Lemma \ref{lem:genericS}, and the index is lower semi-continuous, the result follows.\end{proof}

\begin{remark}\label{rem.the whitney sphere}
Using the Weierstrass representation in $\R^4$ it is an exercise \cite{HoffmanOsserman} to check that every minimal sphere with two planar embedded ends is (up to congruence and $\C$-linear change of coordinates on $\C^2 \cong \R^4$) the \emph{Whitney sphere} 
\[ X_W: z\in \C \mapsto \left(z,\frac{1}{z} \right) \in \C^2 \cong \R^4\,.\]
Furthermore, the inverted Whitney sphere $f_W:=\frac{X_W}{|X_W|^2}$ is the minimizer of the Willmore energy in its regular homotopy class. This follows from the fact that
the two ends of the Whitney sphere invert to a unique transverse double point, and so its normal bundle has Euler characteristic $2$ (see also\ \cite{CastroUrbano, Montiel2}).
This minimizing property implies $\Ind(f_W) = 0$ as deformations preserve the regular homotopy class.
Hence, Corollary~\ref{cor.index estimate in higher codimension} holds true without the restriction on the number of ends. 
\end{remark}

\begin{corollary} \label{minusthree}
 Let $f:\Sp^2 \to \R^3$ be an unbranched Willmore sphere in $\R^3$ that is not the round sphere. Then its Willmore Morse index is 
 \begin{align*}
  \Ind(f) = \frac{1}{4\pi} \W(f) -3.
 \end{align*}

\end{corollary}

\begin{proof}
 By work of Bryant \cite{BryantDuality} we know that a Willmore sphere in $\R^3$ is an inverted complete minimal surface with $m \in \N$ embedded planar ends and $\W(f) = 4\pi m$, $m\geq 4$. We use \cite{HirschEMB}, where the Morse index of a $\W$-critical sphere $f$ is computed to be $\Ind(f) = m -d$ for $d=\dim \Span \{\n(p_1),..., \n(p_m)\}$. Use Corollary~\ref{thm:span at the ends in R3} to get the claimed formula.
\end{proof}

Since the orientation double cover of a real projective plane is a sphere, we can use the above results to compute the Willmore index of any Willmore $\R P^2$ in $\R^3$. Note that Bryant's result \cite{BryantDuality} implies that a Willmore $\R P^2$ is an inverted minimal surface with $m$ embedded planar ends. For $n=3$, we have $m\geq 3$ because an immersed $\R P^2$ with codimension $1$ must \cite{Banchoff} have a point of multiplicity at least $3$.  In our case, there is a unique point of maximal multiplicity which equals the number of ends of the corresponding complete minimal surface, and there are examples for $m =2k+1$, $k\in\N$ \cite{Kusner}, but the existence of a complete minimal $\R P^2$ in $\R^n$ with an even number of embedded, planar ends is unknown.\\[-0.2cm]

Our index formula for spheres also holds for real projective planes.

\begin{proposition} \label{Prop:RP2}
 Let $f:\R P^2\to\R^3$ be an unbranched Willmore real projective plane. Then its Willmore index is 
  \begin{align*}
  \Ind(f) = \frac{1}{4\pi} \W(f) -3.
 \end{align*}
\end{proposition}
\begin{proof}
The orientation double cover of an unbranched Willmore real projective plane $f:\R P^2\to\R^3$ is an unbranched Willmore sphere $\tilde{f}:\Sp^2\to\R^3$ that is invariant under the antipodal map. Using a complex coordinate $z$ on $\Sp^2$, we can express this by
\begin{align*}
   & f :\Sp^2 \diagup \langle I \rangle \to \R^3, \\
   & I (z) = -\frac{1}{\bar z} \ \ \ \text{is the antipodal map}, \\
  & \tilde{f} \circ I = \tilde{f}
\end{align*}
because the antipodal map is (up to M\"obius transformation) the only fixed-point-free anti-holomorphic involution on $\Sp^2$. Bryant's result \cite{BryantDuality} tells us that the Willmore sphere $\tilde{f}:\Sp^2\to\R^3$ is an inverted minimal surface that can be parametrized conformally as $\tilde X:\Sp^2\setminus\{p_1,...,p_m\}\to\R^3$ where the $m$ embedded planar ends correspond to $p_1,...,p_m$. As $\tilde{f}$ is a  double covering map, $m=2k$ must be even and $\W(f) = \frac{1}{2}\W(\tilde{f}) = 4\pi k$. Corollary~\ref{minusthree} shows that the dimension of the space of variations reducing the Willmore energy of $f$ to second order equals $2k -3$. The Willmore index of $f$ is the dimension of the subspace of variations $\tilde u\tilde\nu$ of $\tilde{f}$ covering a variation of $f$. The orientation reversing property of $I$ implies that the normal $\tilde\nu$ of $\tilde{f}$ satisfies $\tilde\nu\big|_{z} = -\tilde\nu\big|_{I(z)} $, so in order to cover $f$ the function $\tilde u$ must satisfy $\tilde u(z) = -\tilde u(I(z))$. We call these variations \emph{odd}. The \emph{even} variations of $\tilde{f}$ satisfy  $\tilde u(z) =\tilde u(I(z))$. \\

Proposition~4.4 from \cite{HirschEMB} provides a more detailed description of $\lambda$-eigen\-functions of $Z$ on $\tilde{f}$ for $\lambda<0$: The space of eigen\-functions is in bijective correspondence to the nontrivial kernel of the linear map $B:\R^{2k} \to \R^3$, $B(v) = \sum_{i=1}^{2k} v_i \tilde\nu(p_i)$, where the $v_i$ are used as values of the variations at the preimages of the point of maximal multiplicity, $v_i= \tilde u(p_i)$. The domain $\R^{2k}$ splits as a direct sum into the subspaces of end-values of the even functions and of the odd functions. Each of these subspaces has dimension $k$. The first subspace is spanned by $v_i=\tilde u(p_i) = \tilde u(I(p_i)) = \tilde u (p_{i+k}) =v_{i+k}$, and $B$ sends it to the zero subspace. The second subspace is spanned by $v_i = - v_{i+k}$ and $B$ maps it onto $\R^3$ because $\tilde\nu$ is odd.
\end{proof}

\begin{corollary}
Let $X: \R P^2\setminus\{p_1,...,p_m\} \to\R^3$ be a complete minimal real projective place with finite total curvature and embedded planar ends at $\{p_1,...,p_m\}$. If $\nu (p_i)$ denotes one of the two possible normals of $X$ at each end $p_i$, then
 \begin{align}
  d: = \dim (\Span \{\n(p_1),..., \n(p_m)\} )= 3.
 \end{align}	
\end{corollary}

\begin{proof}
Note that \cite{Banchoff} implies $m\geq 3$. The statement follows from Corollary~\ref{thm:span at the ends in R3} because the span of the normals of $X$ and of its orientable double cover agree.
\end{proof}

\begin{remark}
It follows that the normals at the ends of a minimal real projective plane with codimension $1$ span $\R^3$. For the higher codimension case, however, the techniques from Lemma~\ref{lem:genericS} do not seem suitable for the case of real projective planes in $\R^n$: the $SO(n,\C)$-action does not commute with the order-two deck transformation on the orientable double cover.
%\medskip

Since the number of ends influences the Willmore Morse index, it is natural to ask whether a minimal real projective plane with two embedded planar ends in $\R^n$ exists. 
We answer this question in the following proposition.
\end{remark}

\begin{proposition}
	There is no immersed minimal $\R P^2$ in $\mathbb{R}^n$ with two embedded planar ends.
\end{proposition}
Our proof of this proposition is technical and is given in Appendix~\ref{appendix:non-existence}.

\section{Asymptotic tangent planes} \label{Section3}
A curator of the mathematical art exhibition {\em Getting to the Surface} \footnote{The exhibition traveled around the world during 1986-87 under the auspices of the US Information Agency after premiering at the Library of Congress in Washington DC.}  dubbed the symmetric complete minimal surfaces described in \cite{Kusner} the \emph{minimal flowers}. They have $m=2p$ embedded planar ends and ambient dihedral symmetry generated by half-turns about a pair of lines at angle $\frac{\pi}{p}$ in $\R^{3}$ when $p$ is even, and $\frac{2\pi}{p}$ when $p$ is odd (in which case the surface is a $2$-fold covering space of a punctured real projective plane). For $p=2$, the inverted minimal flower is an example of a \emph{Morin surface}~\cite{Morin78}.

\begin{definition}
 Let $X:\Sigma \setminus\{p_1,...,p_m\} \to\R^3$ be a complete minimal surface with embedded planar ends. If the $m$ asymptotic affine planes meet at a single point then we call the surface \emph{spiny}. 

\end{definition}

\begin{theorem}\label{thm:flowersspiny}
 Let $X:\Sp^2 \setminus\{p_1,...,p_m\} \to\R^3$ be a complete minimal sphere with embedded planar ends. If $X$ lies in the $\Sp^1 \times SO(3,\C)$-orbit of a minimal flower, then it is spiny.
\end{theorem} 

\begin{proof}
% Idea (copied from one of Rob's Emails):
Suppose $X$ and $X^\ast$ denote the $m$-fold dihedrally symmetric minimal flower and its conjugate surface, that is, the real and imaginary parts of the holomorphic null curve $\Phi$ in $\C$ that can be found in \cite{Kusner}.
 We will repeatedly use the fact that the real and
imaginary parts of the $\C$-tangent vector to any complex null curve provide a basis for the tangent plane for the minimal surface (and for its conjugate minimal surface) with planar ends.\\

From the $m$ half-turn symmetries of $X$, we know that all its asymptotic planes pass through
the origin of $\R^3$, so $X$ is spiny. Its conjugate $X^\ast$ has the conjugate symmetries, that is, reflections in planes at angle $\frac{\pi}{p}$, each of which is perpendicular to corresponding half-turn axis, and so it too has the property that all its asymptotic planes pass through the origin of $\R^3$. 
Thus, $\Phi=X+iX^\ast$ has the property that all its asymptotic complex lines pass through the origin of $\C^3$.
This implies that the entire $\Sp^1$ associate family of surfaces $X_t=\Re(e^{it}\Phi )$ has the property of being spiny.\\

It remains to understand the asymptotic planes of the $\operatorname{SO}(3,\C)$ orbit of $X$. Let $L$ be an asymptotic complex line of $\Phi$ and $M \in \operatorname{SO}(3,\C)$, then the corresponding asymptotic complex line on $M\Phi$ is $ML$. This implies that the asymptotic complex lines of $M\Phi$ pass through the origin of $\C^3$. Thus, for any $M \in \operatorname{SO}(3,\C)$, the minimal surface $X_M=\Re(M \Phi)$ has all its asymptotic tangent planes passing through the origin of $\R^3$. 
\end{proof}

\begin{corollary}\label{cor:16pi}
 Let $f:\Sp^2\to \R^3$ be a compact immersed Willmore sphere with $\W(f) = 16 \pi$. Then the inverted minimal surface $X$ is spiny. 
\end{corollary}

\begin{proof}
 The moduli space of complete minimal spheres with four embedded planar ends consists of the $\Sp^1 \times SO(3,\C)/SO(3,\R)$ orbit of the inverted Morin surface \cite{KusnerSchmitt}. By the previous theorem they are all spiny.
\end{proof}

\begin{lemma} \label{lemma:spiny}
 After a translation, on every spiny minimal surface $X$ the support function $u(p): = X \cdot \nu$ is a bounded Jacobi field for the area functional with the additional property that $u(p_i)=0$ for all ends $p_1,...,p_m$.
\end{lemma} 

\begin{proof} 
  Apply a translation so that the asymptotic planes of $X$ pass through the origin. 
  Recall that the support function $u$ arises by dilation of $X$, see also \cite[Section~5]{PerezRos}. The family $X_t: = (1+t) X$, $t\in (-\epsilon, \epsilon)$ is a smooth deformation of $X$. We note that $0=H(t)$ is the mean curvature of $X_t$. Observe that
 \begin{align*}
  u= \frac{d}{dt}\big|_{t=0} X_t \cdot \nu,
 \end{align*}
and so $0=\frac{d}{dt}\big|_{t=0} H(t) = \frac{1}{2}L u$, where $L$ is the Jacobi operator of $X$. Since $X_t$ is a deformation that is tangential at the ends $p_1,...,p_m$, we have $u (p_i)=0$.
\end{proof}

With this geometric property of being spiny in hand, we provide a simpler proof of the fact that all Willmore spheres with $16\pi$ energy have $\W$-Morse index $1$ (see \cite{HirschEMB}).\\

For the following we need to fix some notation: Let $X:\Sigma\setminus\{p_1,...,p_m\}\to\R^3$ be a complete minimal surface with $m$ embedded planar ends, $g = X^\ast \delta_{\R^3}$ the pullback metric on $\Sigma$. The $2$-manifold $\Sigma$ is compact and without boundary. As in \cite[Section~3]{HirschEMB} we fix local conformal coordinates $z$ in a neighbourhood $U$ around an each end $p_i$ such that $z(p_i) = 0$ for $ i =1,...,m$. We choose the coordinates in such a way that 
\begin{align} \label{choice:coor}
 X_z dz = \left(-\frac{a}{z^2} + Y(z) \right) dz,
\end{align}
where $a\in \C^3, a^2 = 0, |a|^2 = 2$ and $Y(z)$ is holomorphic and bounded.
We denote by $D_{\epsilon}(p_i):= z^{-1}(B_\epsilon)$ the preimage of the Euclidean ball $B_\epsilon \subset \R^2$ under the coordinate $z$.\\

In the following, we use the notation and some results from P{\'e}rez and Ros \cite[Section~5]{PerezRos}. We denote the extended Gau\ss\ map of $X$ by $\nu:\Sigma \to \Sp^2$. The Jacobi operator of the area functional on $X$ is $Lu = \Delta u + |\nabla \nu|^2 u$. We ``compactify'' $L$ by multiplying $g$ with a conformal factor $\lambda \in C^\infty(\Sigma \setminus\{p_1,...,p_m\})$, $\lambda>0$, that decays like $|z|^4$ around each end in the conformal charts: $\hat g = \lambda g$. We can, for example, take $\lambda$ so that $\hat g$ agrees with the pullback metric of the inverted surface $\frac{X}{|X|^2}$ (assuming that $0\not \in X(\Sigma)$). The decay of the conformal factor was chosen to extend $\hat g$ onto $\Sigma$, see (5) in \cite{PerezRos}. The Jacobi operator transforms like $\hat L = \frac{1}{\lambda} L$ and it is now a Schr\"odinger operator on $\Sigma$. \\
 Let $\mathcal B \subset C^{2,\alpha}(\Sigma\setminus\{p_1,...,p_m\})$ be the space of functions $u$ that have the form $u(z) = \beta_i \ln|z| + \tilde u_i(z)$ in the local conformal charts around each end $p_i$, where $\beta_i \in \R$ and $\tilde u_i \in C^{2,\alpha}(D_\epsilon (p_i))$. There is a natural way to equip $\mathcal B$ with a norm so that $\mathcal B$ is a Banach space, see \cite[Section~5]{PerezRos}. We define
 \begin{align*}
  \mathcal J & := \operatorname{kernel} \hat L\\
  \mathcal K &:= C^{2,\alpha}(\Sigma) \cap \mathcal J\\
  \mathcal K_0 &: = \hat L (\mathcal{B})^\perp,
 \end{align*}
where $V^\perp$ denotes the $L^2$-orthogonal of $V$ with respect to $\hat g$. Note that $\mathcal K$ is the (classical) space of (bounded) Jacobi fields on $X$. P\'{e}rez and Ros proved the following statements about these spaces:
\begin{lemma}[Lemma~5.2 in \cite{PerezRos}] \label{lemma:PerezRos}
 With the above notations, $\mathcal J, \mathcal K$ and $\mathcal K_0$ satisfy the following properties:
 \begin{enumerate}
  \item $\mathcal K_0 = \{u \in \mathcal K: u(p_i) = 0, 1\leq i\leq m\}$
  \item $\dim \mathcal J = m +\dim \mathcal K_0$.
 \end{enumerate}
\end{lemma}

\begin{theorem}
 Let $X:\Sigma\setminus\{p_1,...,p_m\} \to\R^3$ be a complete minimal surface with $m$ embedded planar ends and $0\not \in \operatorname{image}(X)$. If there exists a logarithmically growing Jabobi field $u$ on $X$, i.e.\ $\dim \mathcal J- \dim \mathcal K >0$, then the $\W$-index of the inverted and compactified Willmore surface $f: = \frac{X}{|X|^2}$ is at least $1$.
\end{theorem}

\begin{proof}
 The condition $\dim \mathcal J- \dim \mathcal K >0$ means that there exists a function $u \in C^{2,\alpha}(\Sigma\setminus\{p_1,...,p_m\})$ which satisfies the following properties:
 \begin{itemize}
     \item $L u=0$,
     \item $u(z) = \beta_i \ln|z| + \tilde u_i(z)$ locally around $ p_i$, for smooth functions $\tilde u_i\in C^{2,\alpha}(D_\epsilon(p_i))$ and $\beta_i \in \R$,
     \item $\exists i  : \beta_i \not = 0$.
 \end{itemize}
 We use the notation $\beta:= (\beta_1,...,\beta_m)$. \\
 If we knew that $\sum_i\beta_i \not =0$, then a version of the formula for the second variation of the Willmore functional would give us at least one direction that decreases the Willmore energy of the inverted surface $\frac{X}{|X|^2}$, see \cite[Theorem~3.4]{HirschEMB}. For that, we would define $v:= v_0 + \frac{u}{|X|^2}$ and compute $\delta^2 \W (v,v)$. But we cannot guarantee the condition $\sum_i\beta_i \not =0$ a priori. Therefore, we need to ``replace'' $v$ by a suitable function constructed out of $u$, where we ``switch the sign'' at one end. This is done already in Lemma~4.5 and Corollary~4.6 in \cite{HirschEMB}. For the convenience of the reader we repeat the proof here.\\
 For the logarithmically growing Jacobi field $u$ we construct a function $\chi$ with the following properties: $\frac{\chi}{\abs{X}^2} \in W^{2,2}(\Sigma, d\mu_{\hat{g}})$ and  
\begin{itemize}
\item $\abs{\Delta_g \chi}^2 =16$ on $\Sigma$;
\item around each end $p_i$ it has in local conformal coordinates an expansion of the form \[ \chi(z) = \sigma_i \abs{X(z)}^2 + \tilde \chi_i(z)\] for some $\sigma_i \in \{-1,+1\}$ and some smooth $\tilde \chi_i$;
\item if ${\sigma}=(\sigma_1, \dotsc, \sigma_m)$ then ${\sigma}\cdot \beta \neq 0$.
\end{itemize}
This is proven in \cite[Lemma~4.5]{HirschEMB}. We repeat the most important steps: If $\sum_i \beta_i \not =0$, then we are done by defining $\chi : = |X|^2$. If $\sum_i \beta_i  =0$, then we proceed as follows: We can assume $\beta_1 \not =0$. Then we will construct a function $\chi$ with the desired properties which has $\sigma = (-1,1,...,1)$. Then we have that $\sigma\cdot \beta \not = 0$. To construct such $\chi$, first find a function $w$ that agrees with $-|X|^2$ in $D_{\frac{\epsilon}{2}}(p_1)$ and with $|X|^2$ on $\Sigma\setminus D_{\epsilon}(p_1)$ by multiplication with an appropriate cut-off function. Then we define $\eta(x) := 4(1 - 2\cdot \mathbf{1}_{D_{\epsilon}(p_1)})(x)$ and $f(x): = \Delta_g w(x) -\eta(x)$. The function $f$ was constructed in a way that it is supported in the annulus $D_{\epsilon}(p_1) \setminus D_{\frac{\epsilon}{2}}(p_1)$ and it is bounded. The same holds for the function $\frac{1}{\lambda} f$, where $\lambda$ is such that $\hat g = \lambda g$.\\
By Lax-Milgram on the compact surface $\Sigma$, the operator \\
$\Delta_{\hat g} = \frac{1}{\lambda}\Delta_g: W^{2,2}(\Sigma, d\mu_{\hat g}) \to L^2(\Sigma)$ is onto. Let $h \in W^{2,2}(\Sigma, d\mu_{\hat g})$ be the solution of $\Delta_{\hat g} = \frac{1}{\lambda }f$. Then $\chi : = w - h$ has all the desired properties.\\[0.2cm]

The final step is to compute the second variation of $\W$.
Define $w_t:=\chi + t u \in W^{2,2}(\Sigma,d\mu_{\hat g})$ and $v_t: = \frac{w_t}{|X|^2}$. Note that $\chi = \sigma_i |X|^2 + \tilde\chi_i$ and $u = \beta_i \ln|z| + \tilde u_i$  for smooth $\tilde\chi_i, \tilde u_i$ around each end $p_i$. Thus, we use the formula for the second variation of $\W$ in the direction $v_t$ of the inverted surface $f:=\frac{X}{|X|^2}$ from Theorem~3.3  in \cite{HirschEMB}. For every small $\epsilon>0$ the formula reads
\begin{align*}
 \delta^2 \W(f)(v_t,v_t) &=   \int_{\Sigma\setminus \cup_{i=1}^m D_\epsilon (p_i)} \tfrac{1}{2} \left(L w_t\right)^2 d \mu_g  \\
  & \qquad \ \ -  2\sum_{i=1}^m \sigma_i^2\int_{\partial D_{\epsilon}(p_i)} \frac{\partial \abs{X}^2}{\partial \nu} + 8\pi t \sigma\cdot \beta + R_\epsilon  \\
  & = \int_{\Sigma\setminus \cup_{i=1}^m D_\epsilon (p_i)} \tfrac{1}{2} \left(L \chi\right)^2  - 2\Delta_g|X|^2 d \mu_g  \\ 
  & \qquad \ \ + 8\pi t \sigma\cdot \beta + R_\epsilon\\
  & =  \int_{\Sigma\setminus \cup_{i=1}^m D_\epsilon (p_i)} \tfrac{1}{2} \left(\Delta_g \chi + |\nabla\nu|^2_g \chi\right)^2  - \tfrac{1}{2} |\Delta_g\chi|^2 d \mu_g  \\
  & \qquad \ \ + 8\pi t \sigma\cdot \beta + R_\epsilon \\ 
  & =  \int_{\Sigma\setminus \cup_{i=1}^m D_\epsilon (p_i)} \left(\Delta_g \chi + 
  \tfrac{1}{2}|\nabla\nu|^2_g \chi\right) |\nabla\nu|^2_g \chi\, d \mu_g  \\
  & \qquad \ \ + 8\pi t \sigma\cdot \beta + R_\epsilon.
\end{align*}
As $|\nabla\nu|^2_g = -2K_g$ is decaying very fast towards the ends (see (9) in \cite[Section~3]{HirschEMB}) and as the error term $R_\epsilon$ converges to zero, we can pass to the limit $\epsilon \to 0$. The result is
\begin{align*}
 \delta^2 \W(f)(v_t,v_t) &= \int_{\Sigma} \left(\Delta_g \chi + 
  \tfrac{1}{2}|\nabla\nu|^2_g \chi\right) |\nabla\nu|^2_g \chi\, d \mu_g   + 8\pi t \sigma\cdot \beta\\
  & = c_{\chi} + 8\pi t \sigma\cdot \beta.
\end{align*}
We have constructed $\chi$ so that $ \sigma\cdot \beta\not =0$ which implies $\delta^2 \W(f)(v_t,v_t) <0$ for a suitable $t\in \R$. 
\end{proof}

 \begin{corollary}
  Let $f:\Sp^2 \to\R^3$ be a $\W$-critical sphere with $\W(f)=16\pi$. Then the $\W$-index of $f$ is $1$.
 \end{corollary}
\begin{proof}
 After inversion of $f$ we get a complete minimal sphere $X$ with $4$ embedded planar ends. It was proven by Montiel and Ros \cite[Theorem~25]{MontielRos} that such a surface has area-index $4$ and area-nullity $5$, i.e.\ $\dim \mathcal K=5$. Corollary~\ref{cor:16pi} shows that $X$ and its conjugated minimal surface $X^\ast$ are spiny. We use Lemma~\ref{lemma:spiny} to see that the support function $u = X\cdot \nu$ of $X$ and the support function $\bar u = X^\ast\cdot \nu$ of $X^\ast$ are linearly independent bounded Jacobi fields on $X$ such that $u(p_i)=0=\bar u(p_i), i=1,...,m$. Note that we have used that $X$ and $X^\ast$ have the same Gau\ss\ map which implies that $\bar u$ is a Jacobi field of $X$. A way of seeing this is that the family $X(t) = \Re (\int e^{it} \Phi)$ for the null curve $\Phi= X+ iX^\ast$ is a deformation of $X$ by complete minimal surfaces. As a consequence, $\bar u = \frac{d}{dt} X(t)\big|_{t=0} \cdot \nu= - X^\ast\cdot \nu$ is a Jacobi field of $X$, see \cite{PerezRos}. The mentioned deformation $X(t)$ was studied in \cite{RosenbergToubiana}. Using Lemma~\ref{lemma:PerezRos} (by P\'{e}rez and Ros)  we have that $\dim \mathcal K_0 =2$ and
 \begin{align*}
  \dim \mathcal J - \dim\mathcal K = 6-5=1.
 \end{align*}
The previous theorem shows that the $\W$-index of $f$ is at least one. In \cite[Proposition~1.7]{HirschEMB} we showed the upper bound $\Ind(f) \leq m-d = 1$.
\end{proof}

\begin{remark}
              The above approach uses the ``spininess'' of the surface. Theorem~\ref{thm:flowersspiny} shows that the minimal flowers and certain deformations of them are spiny. To the best of our knowledge, it is unclear whether all minimal surfaces with embedded planar ends in $\R^3$ are spiny. In the recent work \cite{Lee} it is shown that there are minimal surfaces with embedded planar ends in $\R^4$ that are \underline{not} spiny. A minimal sphere with three embedded planar ends is constructed in $\R^4$ where two of the three ends are parallel. This implies that the three asymptotic planes have an empty intersection. (The article \cite{Lee} also contains several interesting existence and non-existence results about minimal surfaces with embedded planar ends.)
\end{remark}

\section{The density at infinity} \label{Section4}
In the following we explain how the formula for the second variation of $\W$ on an inverted minimal surface with embedded planar ends can be used to refine the notion of \emph{density at infinity} for that minimal surface.

\begin{lemma}\label{lemma:L}
 Let $X:\Sigma\to\R^3$ be a complete minimal surface, $g = X^\ast \delta$ the pullback metric on $\Sigma$, $\nu_X$ the unit normal along $X$ and $f = i\circ X = \frac{X}{|X|^2}$ with corresponding normal $\nu_f= \nu_X - 2\left(X\cdot \nu_X\right) \frac{X}{|X|^2}$ along $f$. Let $L_g \varphi = \Delta_g \varphi - 2 K_g \varphi$ be the Jacobi operator (for the area functional) of $X$. Then we have that
 \begin{align}
  L_g\left(|X|^2 \nu_f^k\right) = 4 \nu_X^k \ \ \text{ for } k=1,2,3,
 \end{align}
where $\nu_f^k =  \nu_f\cdot e_k$ and $\nu_X^k =  \nu_X \cdot e_k$.
\end{lemma}

\begin{proof}
For $i(x) = \frac{x}{|x|^2}$ and $f = i\circ X$ the formula  
\begin{align}\label{eq:normal}
 \nu_f = \nu_X - 2\left(X\cdot \nu_X\right) \frac{X}{|X|^2}
\end{align}
%follows from $D i(x) = \frac{1}{|x|^2}\left(Id - 2\frac{x\otimes x}{|x|^2}\right)$.
provides a unit normal along $f$. The choice of the sign of $\nu_f$ is consistent with the sign choice in \cite{Michelat} and \cite{HirschEMB}.
Using (\ref{eq:normal}) we get that
\begin{align*}
 L_g\left(|X|^2  \nu_f^k\right) &=L_g\left(|X|^2  \nu_X^k\right) - 2L_g\left(\left( X\cdot \nu_X\right) X^k\right)=: I - 2\, II.
\end{align*}
In the following, we compute everything in a conformal parametrization over $U\subset\mathbb C$. We can choose for example the Weierstra\ss\ representation. We have that $g=e^{2\lambda }\delta$ and $L_g \varphi = e^{-2\lambda} \Delta \varphi - 2 K_g \varphi$ for the Euclidean Laplacian $\Delta$. As $X$ is minimal and conformal, we know that $\Delta X^k=0$, $L_g \nu_X^k=0$ and $\Delta_g |X|^2=4$. We compute
\begin{align*}
 I & = |X|^2 \Delta_g \nu_X^k + 2 e^{-2\lambda}\nabla \nu_X^k \cdot \nabla |X|^2 + 4 \nu_X^k - 2K_g \nu_X^k |X|^2\\
 &= 4 \nu_X^k + 4 e^{-2\lambda} \partial_j \nu_X^k \partial_j X^l X^l,
\end{align*}
where we use summation convention for doubly-repeated indices. The second term gives
\begin{align*}
 II& = \left(\Delta_g \left(X\cdot \nu_X\right)\right) X^k + 2 e^{-2\lambda} \nabla \left(X\cdot \nu_X\right) \cdot \nabla X^k - 2 K_g \left(X\cdot \nu_X\right) X^k\\
 &= X^k \left\{ X^l \Delta_g \nu_X^l  + 2 e^{-2\lambda} \nabla X^l \cdot\nabla \nu_X^l\right\}\\
 &\qquad\qquad\qquad\qquad\quad +2 e^{-2\lambda} \nabla \left(X\cdot \nu_X\right) \cdot \nabla X^k - 2 K_g \left(X\cdot \nu_X\right) X^k\\
 &= 2 e^{-2\lambda} \nabla X^l \cdot\nabla \nu_X^l\, X^k+ 2 e^{-2\lambda} \nabla X^l \cdot \nabla X^k \nu_X^l +  2 e^{-2\lambda} \nabla \nu_X^l  \cdot \nabla X^k \, X^l \\
  &= 2 e^{-2\lambda}\left(0 + \partial_j X \cdot \nu_X\, \partial_j X^k  +  \nabla \nu_X^l  \cdot \nabla X^k\, X^l \right)\\
  & = 2 e^{-2\lambda}\nabla \nu_X^l  \cdot \nabla X^k\,X^l,
\end{align*}
where we used 
\begin{align*}
 e^{-2\lambda} \nabla X^l \cdot\nabla \nu_X^l = e^{-2\lambda}\partial_j\left(\partial_j X \cdot \nu_X\right) - \nu_X \cdot \Delta_g X = 0.
\end{align*}
As $\partial_j \nu_X  = \left(\partial_j \nu_X\right)^T = e^{-2\lambda} \partial_j \nu_X \cdot \partial_m X\, \partial_m X$ we get that $\partial_j \nu_X^s  =- e^{-2\lambda} \nu_X \cdot  \partial_m\partial_j X\, \partial_m X^s$. We collect all terms and get
\begin{align*}
 L_g &\left(|X|^2  \nu_f^k\right)  = I -2 \,II\\
 & = 4 \nu_X^k + 4 e^{-2\lambda} X^l\left\{ \partial_j \nu_X^k \partial_j X^l  - \partial_j \nu_X^l \partial_jX^k \right\}\\
 & =  4 \nu_X^k - 4 e^{-4\lambda} X^l\left\{   \nu_X \cdot  \partial_m\partial_jX\, \partial_m X^k\, \partial_jX^l -   \nu_X \cdot \partial_m\partial_j X\, \partial_m X^l  \partial_jX^k \right\}\\
 & =  4 \nu_X^k.
\end{align*}

\end{proof}

\begin{corollary}\label{cor:dens}
 Let $\Sigma$ be a closed surface and $X:\Sigma\setminus\{p_1,...,p_m\}\to \R^3$ a complete minimal immersion with $m$ embedded planar ends at $p_1,...,p_m$. With the choice of the coordinates described in (\ref{choice:coor}) we have that 
 \begin{align}\label{eq:dens}
  \lim_{\epsilon\to 0}\left\{ \int_{\Sigma_\epsilon} d \mu_g -  \frac{\pi\, m}{\epsilon^2}\right\} = 0, 
 \end{align}
where $\Sigma_\epsilon : = \Sigma \setminus \bigcup_{i=1}^m D_\epsilon(p_i)$.
\end{corollary}
\begin{proof}
We translate the surface $X$ such that $0 \not\in X(\Sigma)$. 
 Consider the smooth Willmore surface $X: = \frac{X}{|X|^2} = i\circ X$. Any component of the normal $\nu_f^k$ is a smooth $\mathcal W$-Jacobi field because the Willmore functional is translation invariant. We use the formula for the second variation of $\W$ for a smooth variation proven by Michelat \cite[Theorem~4.7]{Michelat}. Note that a term with the notation $\operatorname{Res}_{p}(X, U)$ appears in his formula (see Definition-Proposition~4.6 in \cite{Michelat}). We have scaled our coordinates around an end in such a way that $\operatorname{Res}_{p}(X, U) =2$ (see also Remark after Lemma~3.1 in \cite{HirschEMB}). We also use Lemma~\ref{lemma:L} to get that
 \begin{align*}
  0 &= \delta^2 W (f)(\nu_f^k,\nu_f^k) = \lim_{\epsilon\to 0} \left\{ \frac{1}{2} \int_{\Sigma_\epsilon} \left(L\left(|X|^2 \nu_f^k\right)\right)^2 d\mu_g - \sum_{i=1}^m \frac{8\pi}{\epsilon^2} \left(\nu_f^k(p_i)\right)^2\right\}\\
 &= \lim_{\epsilon\to 0} \left\{ 8 \int_{\Sigma_\epsilon} \left( \nu_X^k\right)^2 d\mu_g - \sum_{i=1}^m \frac{8\pi}{\epsilon^2} \left(\nu_f^k(p_i)\right)^2\right\}.
 \end{align*}
We sum this equation for $k=1,...,3$ and get that
\begin{align*}
  0 &= \lim_{\epsilon\to 0} \left\{ 8 \int_{\Sigma_\epsilon} d\mu_g -  \frac{m\,8\pi}{\epsilon^2} \right\}.
\end{align*}

\end{proof}

\begin{remark}
 Corollary~\ref{cor:dens} shows that $\epsilon^2 |X\left(\Sigma\setminus \bigcup_{i=1}^m D_{\epsilon}(p_i)\right)|= m\,8\pi +o(\epsilon^2)$.
 Whereas the monotonicity formula for minimal surfaces implies that 
 \[\epsilon \mapsto \epsilon^2|X(\Sigma)\cap B_{\frac1\epsilon}|\]
 is non-increasing and converges to $m\pi$.
 Recall that for sufficiently large $R_0>0$, due to our assumption that all $m$ ends are planar and embedded,  we have 
 \[X(\Sigma)\setminus B_{R_0} = \bigcup_{i=1}^m \mathbb{G}_{u_i} \cap B_{R_0}^c\,,\]
 where $\mathbb{G}_{u_i}$ denotes a minimal graph over the $i$th end $p_i$ satisfying the estimates
 \begin{align*}
     u_i(x) &= b_i + \frac{c_i\cdot x}{\abs{x}^2} + O(\abs{x}^{-2})\\
     \abs{\nabla u_i} &= O(\abs{x}^{-2})\,
 \end{align*}
  (compare \cite[Chapter 2, Proposition 1]{Schoen}).
 Hence an easy calculation shows that for $\epsilon<<R^{-1}_0$
 \[\abs{X(\Sigma) \cap B_{\frac1\epsilon}}= \frac{m\pi}{\epsilon^2} + O(R_0^{-2}) +  \abs{X(\Sigma) \cap B_{R_0}}\,.\]
 Thus one can consider Corollary~\ref{cor:dens} as an ``improved convergence'' estimate, if one is allowed to deform the Euclidean ball $B_{\frac1\epsilon}$ appropriately. 

\end{remark}

\section{The equivariant index and the sign of eigenfunctions of the Jacobi operator } \label{Section5}

Here we consider the equivariant index of a Willmore immersion invariant under a cyclic group of isometries. Examples are the minimal flowers from Section~\ref{Section2}. We use the notation and definitions of the index form $\delta^2\W$ and the Jacobi operator $Z$ of the Willmore functional $\W$ (see Definition~\ref{def:index}).

\begin{lemma}\label{lem.isometryinvariance}
Let $S\colon \R^3 \to \R^3$ be an isometry and $f\colon \Sigma \to \R^3$ a $2$-sided Willmore immersion that is not a multiple cover. If $f$ is
invariant under $S$, that is $S(f(\Sigma))=f(\Sigma)$, then %the following holds:
\begin{enumerate}
    \item\label{l.inducedI} $S$ induces an isometry $I\colon \Sigma \to \Sigma$ and
    \item\label{l.commute} the index form $\delta^2\W$ is invariant under the action of $I$, in particular
    \[ Z(u\circ I) = (Zu)\circ I\,.\] 
\end{enumerate}
\end{lemma}
\begin{proof}
\emph{\ref{l.inducedI})} We show that there is a uniquely defined map $I\colon \Sigma \to \Sigma$ such that
\begin{equation}\label{eq.propertyofI}
    S\circ f = f \circ I \text{ and } S \circ \nu = \sigma \nu \circ I
\end{equation}
where $\sigma=1$ if $S$ is orientation preserving and $\sigma=-1$ otherwise.
It is enough to check it locally. Fix any $z \in \Sigma$ and $p=S\circ f (z) \in f(\Sigma)$. Since $f$ is an immersion of a compact surface, we have $f^{-1}(p)=\{w_0, \dotsc, w_m\}$ for some $m<\infty$. If $m=1$, the map $I=f^{-1}\circ S \circ f$ is well-defined. If $m>1$, we may choose  local neighborhoods $z\in U$, $w_i \in V_i$ in $\Sigma$ together with a $2$-dimensional plane $\pi_0$ such that $S\circ f(U)$ and $f(V_i)$ are graphs of functions $u\colon \pi_0 \to \R $, $v_i \colon \pi_0 \to \R$. Since $f$ is a Willmore immersion, the functions $u, v_i$ are analytic. Hence $u$ can only agree with a single $v_{j}$ for some fixed $j$. Therefore $I$ is well-defined for $m>1$
as well.

\emph{\ref{l.commute})} Let $w \in C^\infty(\Sigma)$ be arbitrary and note that for the variations $f_t= f + t(w\circ I)\nu$ and $\hat{f}_t = f+ tw\nu$ we have
\begin{align*}f_t(x)&=f(x) + t w(I(x)) \,\nu(x)\\&= S^{-1}(f(I(x))) + t\sigma  w(I(x)) \, S^{-1}(\nu(I(x)))\\&= S^{-1}(f + \sigma t w \nu)( I(x))= S^{-1}(\hat{f}_{\sigma t}(I(x)))\,.\end{align*}
Since the Willmore energy is invariant under isometries of $\R^3$ and does not depend on the chosen parametrization, we deduce that 
\[\W(f_t)= \W(S^{-1}\circ\hat{f}_{\sigma t}\circ I)=\W(\hat{f}_{\sigma t})\,. \]
Differentiating twice in $t$ give 
\[ \int_\Sigma Z(w\circ I)\,w\circ I\, d\mu_g = \sigma^2 \int_\Sigma Z(w)\,w\,d\mu_g = \int_\Sigma Z(w)\,w \, d\mu_g\,. \]
The polarization identity implies
\[\int_\Sigma Z(v\circ I)\,w\circ I\, d\mu_g = \int_\Sigma Z(v)\,w \, d\mu_g\,,\]
for all $v,w\in C^\infty(\Sigma)$ which proves the claim.
\end{proof}
%}
 \begin{remark}
             The above result holds true as well in the non-orientable and higher codimension situation. The presented argument does neither rely on the orientability nor on the codimension. The argument in \ref{l.inducedI}) essentially carries over to the higher codimension case. In \ref{l.commute}) one only needs to observe that for given any $w \in C^\infty(\Sigma, N(\Sigma))$ the variation $S^{-1}w \circ I$ is admissible as well. Hence, one considers $f_t=f+ t S^{-1}w\circ I$ and gets 
             \[f_t = S^{-1} ( f + t w ) \circ I\,.\]
\end{remark}

\begin{theorem}
Let $f:\Sp^2 \to\R^3$  be a closed, immersed Willmore sphere such that $X:= \frac{f}{|f|^2}:\Sp^2\setminus\{p_1,..., p_{2p}\} \to \R^n, p \in \N$, is a complete, immersed minimal sphere with $2p$ embedded planar ends. Assume further that $f$ has an orientation reversing $2p$-fold rotational symmetry around an axis of symmetry going through $0=f(p_i)=f(p_j)$  for all $i,j\in\{1,...,2p\}$. 

 Under these conditions, the subspace of $p$-fold rotation-symmetric variations of $f$ decreasing $\W$ to second order is $1$-dimensional.
 (If $p$ is prime, all other variations corresponding to negative eigenvalues of $Z$ break the symmetry.) 
 \end{theorem}

\begin{proof}
After a rigid motion, $S$ corresponds to the rotation about the $x_3$-axis by the angle $\frac{\pi}{p}$.
We note that Corollary~\ref{thm:span at the ends in R3} together with \cite[Proposition~2.9]{HirschEMB} yields that if $f_t$ is a smooth family that preserves the $2p$-fold orientation reversing symmetry, we have that
\begin{equation}\label{eq.nonnegative}
	0 \le \frac{d^2}{dt^2} \W(f_t) = \int_\Sigma Z(v)\,v\, d\mu_g,
\end{equation}
where $v \nu=\frac{d}{dt}|_{t=0} f_t$.

Furthermore, we are in the setting of the previous lemma. Hence $S$ induces an orientation reversing isometry $I\colon \Sp^2 \to \Sp^2$ defined by 
\[ S\circ f= f\circ I \text{ and } S\circ \nu= - \nu\circ I\,, \]
and the operators $Z$ and $I$ commute. We conclude that $I$ leaves the eigenspaces of $Z$ invariant: for any $Z(u)=\lambda u$, we have  \[Z(u\circ I) = (Zu)\circ I =  \lambda \,u\circ I \,. \]
Thus, $I$ acts as an orthogonal transformation on the eigenspaces of $Z$. 
Over~$\C$, we can diagonalize $Z$ and $I$ simultaneously. 
Since $I$ is the generator of a cyclic group of order $2p$, its eigenvalues are elements of $\{ e^{i\frac{k\pi}{p}} \colon k \in \N \}$. 
We denote the $\lambda$-eigenspace of $Z$ by $E(\lambda)= \ker( Z-\lambda \text{Id})$. Furthermore, let us order the planar ends $p_1, \dotsc, p_{2p}$ such that $I(p_i)=p_{i+1}$ with $p_1=p_{2p+1}$. 

\emph{Claim 1:} Let $\sigma$ be a real eigenvalue of $I$ on $E(\lambda)$, $\lambda <0$, then $\sigma = +1$. \\
We clearly must have $\sigma \in \{-1,1\}$. Assume by contradiction that $\sigma=-1$, that is $u(I(x)) = - u(x)$. We claim that, under these assumptions, the variation $f_t= f + t u \nu$ preserves the $2p$-fold, orientation-reversing symmetry which leads to a contradiction for $\lambda <0$ due to (\ref{eq.nonnegative}). We calculate 
\begin{align*}
S(f_t(x))&= S(f(x)) + t u(x) S(\nu(x))= f(I(x))- t u(x)	\nu(I(x))\\
&= f(I(x)) + t u(I(x)) \nu(I(x))=f_t(I(x))\,
\end{align*}
to see the preservation of the given symmetry for $\sigma=-1$.

\emph{Claim 2:} $\ker( I - (+1)\text{Id}) \cap \bigcup_{\lambda<0} E(\lambda)$ is at most 1 dimensional. \\
Suppose this is not the case. Then we can find two orthogonal eigenvectors $u_1,u_2$ of $Z$ in $\ker( I - (+1)\text{Id})$. Since $S$ generates the cyclic group, that is $I(p_j)=p_{j+1}$ for all $j$ and $u_i\circ I = u_i$, we have that $u_i(p_j) = u_i(p_1)$ for all $j$. 
We must have $u_i(p_1) \neq 0$ because otherwise the variation $f_t:= f + t u_i \nu$ would not reduce the multiplicity at $p_1$,  which would then create the following contradiction: 
\[0\le \frac{d^2}{dt^2} \mathcal{W}(f_t) = \int_\Sigma Z(u_i)\, u_i \, d\mu_g < 0\,.\]
Let $\alpha := \frac{u_1(p_1)}{u_2(p_1)}$ and $v:=u_1 - \alpha u_2$. By our choice of $\alpha$ we have $v(p_j)=0$ for all $j$. But now, the same argument as above yields a contradiction:
\[0\le \frac{d^2}{dt^2} \mathcal{W}(f_t) = \int_\Sigma Z(u_1)\, u_1  \, d\mu_g + \alpha^2\int_\Sigma Z(u_2)\, u_2  \, d\mu_g  < 0\,,\]
where we used the variation $f_t:=f + t v \nu$. \\[-0.3cm]

\emph{Claim 3: } If $\dim( E(\lambda)), \lambda <0$, is odd, then $\ker(I - (+1)\text{id}) \cap E(\lambda) \neq \{0\}$. \\
We can diagonalize $I$ on $E(\lambda)$ over $\C$. Since the complex eigenvalues of $I$ come in conjugate pairs, there must be at least one real eigenvalue $\sigma$. Now the claim follows by Claim~1.\\[-0.3cm] 

The three claims together prove the statement because the index of $f$ is $2p-3$, an odd integer. This follows from Corollary~\ref{minusthree}.

\end{proof}

\begin{remark} This theorem supports the second author's idea \cite{Kusner} that the compactified minimal flowers arise as $\W$-minmax surfaces for cyclically-symmetric eversions (paths of immersed spheres joining oppositely oriented round spheres).

\begin{enumerate}
               \item It seems to be likely that the variation corresponding to the lowest eigenvalue of $Z$ on $f$ is $p$-equivariant. At this point we cannot prove this statement. The above proof shows that finding a simple negative eigenvalue identifies the most symmetric variation.
               \item The existence of a $p$-equivariant negative variation can also be shown by using representation theory:  A standard result  from representation theory
tells us that every irreducible real representation of a cyclic group is either 1- or 2-dimensional. If it is $1$-dimensional, then it is either trivial (meaning $u=u\circ I $) or it is the ``sign'' representation (meaning $ u\circ I = - u$). Since the $\W$-index of $f$ is odd, we get that there is at least one $1$-dimensional irreducible, real representation corresponding to the isometry $ I$. The ``sign'' representation can be ruled out by the same argument as in Claim~1 in the proof above.
\item The minimal flowers (see Section~\ref{Section3}) found in \cite{Kusner} satisfy the condition of the theorem.
\item The proof shows that also other topological types with the symmetries  described in the theorem above have one $p$-symmetric variation if their $\W$-index is odd.
\item The statement of the theorem above was conjectured by the second author \cite[Remark~3~(ii)]{Kusner} and illustrated in \cite{Francis97}, Section~1.
              \end{enumerate}

\end{remark}
Another interesting question is whether the negative-Hessian variations --- in particular, those corresponding to the lowest eigenvalue --- have a sign change. We can answer this question in the following case.
\begin{proposition}\label{prop.sign}
 Let $f: \Sigma\to\R^3$ be a Willmore surface that arises as a complete minimal surface with $m$ embedded planar ends $X:\Sigma\setminus\{p_1,...,p_m\} \to\R^3$. Let $v\in C^\infty(\Sigma)$ be an eigenfunction corresponding to an eigenvalue of the Jacobi operator $Z$ for $\W$ satisfying the condition $v(p_i) = v_0$ for all $i=1,...,m$. If $v$ is non-negative or non-positive, then $\lambda \geq 0$.
\end{proposition}

\begin{proof} %\label{prop:sign}
 Since the eigenfunctions are smooth on $\Sigma$, we can use the formula of Michelat \cite{Michelat} for the second variation of $\W$. Using also $v(p_i)= v_0$ this formula reads (\cite[Theorem~4.5]{Michelat} and \cite[Remark after Lemma~3.1]{HirschEMB})
 \begin{align*}
  \delta^2 \W(f)(v,v) = \lim_{\epsilon \to 0}\Big\{ \frac{1}{2}\int_{\Sigma_\epsilon} (\Delta_g w - 2 K_g w)^2 d\mu_g - \frac{8\pi m}{\epsilon^2}\Big\},
 \end{align*}
where $w = |X|^2v$, and $\Sigma_\epsilon:= \Sigma\setminus\bigcup_{i=1}^m D_\epsilon (p_i)$ for our choice of conformal coordinates $z$, $D_\epsilon(p_i) := z^{-1}(B_\epsilon)$, see Section~\ref{Section3}. Using $\Delta_g (|X|^2 v_0) = 4v_0$ and (\ref{eq:dens}) we compute
\begin{align*}
  \delta^2 \W(f)(v,v) &= \lim_{\epsilon \to 0}\Big\{ \frac{1}{2}\int_{\Sigma_\epsilon} (\Delta_g (|X|^2(v-v_0)) + 4 v_0 - 2 K_g w)^2 d\mu_g - \frac{8\pi m}{\epsilon^2}\Big\}\\
  & = \lim_{\epsilon \to 0} \frac{1}{2} \int_{\Sigma_\epsilon} (\Delta_g (|X|^2(v-v_0))  - 2 K_g |X|^2v)^2  d\mu_g \\
  & \qquad\qquad + 4 v_0 \int_{\Sigma_\epsilon}   (\Delta_g (|X|^2(v-v_0))  - 2 K_g |X|^2 v) d\mu_g.
\end{align*}
The term $\int_{\Sigma_\epsilon} \Delta_g  (|X|^2(v-v_0)) d\mu_g$ is a boundary term. As we are in the smooth setting, we get that 
\begin{align*}
 \lim_{\epsilon \to 0} \int_{\Sigma_\epsilon} \Delta_g  (|X|^2(v-v_0)) d\mu_g = \lim_{\epsilon \to 0} \int_{\partial \Sigma_\epsilon} \frac{\partial  (|X|^2(v-v_0))}{\partial r} = 0.
\end{align*}
This can for example be seen by doing a Taylor expansion of $v$ at $p_i$ and multiplying this expansion by $|X|^2$. By subtracting $v_0$ from $v$ we eliminated the leading term in $|X|^2(v-v_0)$, which then is only a linearly growing term towards the ends. It was for example shown in \cite[Proof of Theorem~4.5]{Michelat} that only the quadratically growing term in the boundary term has a non-vanishing limit when $\epsilon \to 0$ \underline{in the smooth setting}.\\
Combining the above calculations we get that
\begin{align*}
  \delta^2 \W(f)(v,v) = \lim_{\epsilon \to 0} \frac{1}{2} \int_{\Sigma_\epsilon} (\Delta_g (|X|^2(v-v_0))  - 2 K_g |X|^2v)^2   - 16 v_0 K_g |X|^2 v d\mu_g.
\end{align*}
For a minimal surface we have $K_g \leq 0$, which implies that $ \delta^2 \W(f)(v,v) \geq 0$ if $v$ has a sign. In other words, the eigenvalue $\lambda$ cannot be negative.
\end{proof}

\begin{corollary}
 Let $f:\Sp^2\to \R^3$ be the Morin surface. Then the eigenfunction $v$ of the Jacobi operator $Z$ of $\W$ that corresponds to the negative eigenvalue of $Z$ on $f$ can be chosen such that $v(p_i) = v_0 >0$ for all $i=1,...,4$. Then $v$ must have a sign change.
\end{corollary}
\begin{proof}
The $\W$-index of the Morin surface is $1$ due to either Corollary~\ref{minusthree} or \cite[Theorem~3.5]{HirschEMB}. Let $u\in C^\infty (\Sp^2)$ be the eigenfunction corresponding to the negative eigenvalue. Since $Z$ is self-adjoint we can add elements of the kernel of $Z$ (i.e.\ $\W$-Jacobi fields) to $u$ without changing the second variation: for every $v:= u+ j$, $j\in \operatorname{kernel}(Z)$, we have that $\delta^2 \W(f)( u,u) = \delta^2\W(f)(v,v)$. In \cite[Theorem~2.3]{HirschEMB} we used known $\W$-Jacobi fields on a Willmore spheres -- namely those coming from the translation invariance of $\W$ -- to show the upper bound $\operatorname{Ind}_{\W}(f) \leq m-d$. For the Morin surface, this approach implies that we subtract suitable linear combinations of $j = \langle \nu_f, \vec a\rangle$ for vectors $\vec a\in \R^3$ in order to arrange for $v := u -\sum_{l=1}^3 \langle \nu_f, \vec a_l\rangle$ the property $v(p_1) = v_0 >0$ and $v(p_2)= ... = v(p_4 )=0$. \\
We will now use the symmetry conditions of $f$. It can be computed that the unit normals of $f$ span a regular tetrahedron. In fact, in the parametrization of the second author from \cite{Kusner}, we get that $\nu_f(p_1) = ( c, 0 ,d)$, $\nu_f(p_2) = (- c, 0 ,d)$,  $\nu_f(p_3) = ( 0, c ,-d)$, $\nu_f(p_4) = ( 0, -c ,-d)$ for $c = \sqrt{ \frac{2}{3}} $ and $d = \sqrt{\frac{1}{3}}$. In the following we use $\nu^k_f:= \langle \nu_f, e_k\rangle$ for the coordinate functions of $\nu_f$. Defining $\bar v: = v - \nu_f^1 \frac{v_0}{2c}$ we arranged $\bar v(p_1) = \bar v(p_2) = \frac{v_0}{2}$ and $\bar v(p_3)=0=\bar v(p_4)$. Then we define $\tilde v:= \bar v - \nu^3_f \frac{v_0}{4d}$ which implies $\tilde v(p_i) = \frac{v_0}{4}>0$ for all $i=1,...,4$. We use Proposition~\ref{prop.sign} above and $\delta^2 \W(\tilde v,\tilde v)\geq 0$ to get a contradiction if $\tilde v$ has a sign.  
 \end{proof}

\appendix
\section{Real projective planes with embedded planar ends} \label{appendix:non-existence}
We show a minimally immersed real projective plane with embedded planar ends must have at least $3$ such ends (compare Section~\ref{Section2}).
\begin{proposition}
	There is no minimal $\R P^2$ minimally immersed in $\R^n$ with two embedded planar ends.
\end{proposition}

\begin{proof}
	First observe \cite{Kusner} (or by the extended monotonicity formula \cite{EkholmWhiteWienholtz}) that any (branched) immersed minimal surface in $\R^n$ with finite total curvature and two embedded ends must be embedded.  Moreover, it lies in some $\R^4$ due to Lemma \ref{lem.number of ends 1}. We will use the 
	Weierstra\ss\ representation  on the orientation double cover to rule out a minimal embedding $X\colon \R P^2\setminus\{q_1,q_2\}\to \R^4$. 	

	\emph{Step 1: } \underline{A ``good'' parametrization for $X$}:\\
	Because $\R P^2$ has a unique conformal structure, we may assume the embedding $X\colon \R P^2\setminus\{q_1,q_2\}\to \R^4$ is conformal.
	Thus there is a conformal parametrization $\tilde{X} \colon \Sp^2\setminus \{p_1,p_2,I(p_1), I(p_2)\} \to \R^4$ of the orientation double cover with 
	$\tilde{X}=\tilde{X}\circ I$,
	where the antipodal map $I\colon \Sp^2 \to \Sp^2$ is the orientation reversing order-2 deck transformation. % $I(x)=-x$.} 
	This follows from the fact that (up to M\"obius transformation) $I$ is the only antiholomorphic involution without fixed points on $\Sp^2$, 
	which we identify with $\hat{\C}=\C\cup\infty$ via stereographic projection.
	After rotation we can assume that  $p_1=\infty$ (the north pole). %and $p_2 \in \Sp \cap \{x_2=0\}$. 
	Hence we have a minimal immersion
	\begin{equation}\label{eq.WeierstrassRP2}
		\tilde{X} \colon \C \setminus \{0, p, I(p)\} \to \R^3 \text{ with } \tilde{X}=\tilde{X}\circ I \text{ where } I(z)=-\frac{1}{\bar{z}}\,.
	\end{equation}
	After a further rotation we may assume that $p \in \R_+$.
	
	\emph{Step 2: } \underline{The Weierstra\ss\ representation on the orientation} double cover:\\
	Recall that $\phi=\partial_{z}\tilde{X}\colon \C \to \C^4$ is a meromorphic function satisfying the following properties:
	\begin{enumerate}
		\item $\phi^2\equiv0$;
		\item around each finite pole $p_i$ it has the expansion $\phi(p_i +z) = \frac{a_i}{z^2} + h_i(z)$ and at $\infty$ it satisfies $\phi(z)= a_\infty + \frac{1}{z^2} h_\infty(\frac1z)$ where $h_i,h_\infty$ are holomorphic in a neighborhood of $0$;
		\item $\phi(z)=\overline{\phi \circ I}\; \overline{\partial_{\bar{z}}I}$.
	\end{enumerate}
	Property (i) encodes that $\tilde{X}$ is conformally parametrized, (ii) encodes that each end is embedded and planar, (iii) comes from differentiating $\eqref{eq.WeierstrassRP2}$.
	We deduce immediately: 
	\begin{enumerate}
		\item[0)] Since we have two pairs of poles as in (\ref{eq.WeierstrassRP2}) and $p$ is real
		\[\phi(z)=\frac{b_1}{z^2} +b_2 + \frac{c_1}{(z-p)^2} + \frac{c_2p^2}{(1+p z)^2}\,.\]
		\item[1)] Combining i) and ii) gives 
		\[a_i^2=0, \quad a_i\cdot h_i(0)=0, \quad h_i(0)^2=0 \text{ for all $i$ and $i=\infty$ }\,.\] 
		\item[2)] Combining 0) with iii) gives $b_1=\bar{b_2}$ and $c_1=\bar{c_2}p^2$ hence 
		\[\phi(z)=\frac{b}{z^2} + \bar{b} + \frac{\bar{c}p^2}{(z-p)^2} + \frac{cp^2}{(1+p z)^2}\,.\]
		\end{enumerate}
		In particular we deduce that 
		\begin{align*}
			h_0(0)&=\bar{b} + (\bar{c}+ p^2c)\\
			h_p(0)&=\frac{cp^2}{(1+p^2)^2}+(\frac{b}{p^2}+\bar{b})=\frac{1}{p^2}\left(\frac{cp^4}{(1+p^2)^2}+(b+\bar{b}p^2)\right)
		\end{align*}
		Due to iii) we have $h_\infty(0)=\overline{h_0(0)}$ and $\overline{h_{I(p)}(0)}=p^2h_p(0)$\,.\\
		Property i) implies $b^2=0, c^2=0$. We use this together with
		\[(\bar{c}+ p^2c)^2=2|c|^2p^2 \text{ and } (b+\bar{b}p^2)^2=2|b|^2p^2\,\]
		to get these four equations
		\begin{align}
			\label{eq.condition1}0&=\frac12(h_0(0))^2=\bar{b}\cdot (\bar{c}+ p^2c)+|c|^2p^2\\
			\label{eq.condition2}0&= \frac{p^4}{2}(h_p(0))^2=\frac{p^4}{(1+p^2)^2} c\cdot(b+\bar{b}p^2) + |b|^2p^2 \\
			\label{eq.condition3}0&=b\cdot h_0(0)=|b|^2 + b\cdot (\bar{c}+ p^2c)\\
			\label{eq.condition4}0&=\bar{c}\cdot h_p(0)=\frac{1}{p^2}\left(\frac{|c|^2p^4}{(1+p^2)^2}+ \bar{c}\cdot(b+\bar{b}p^2)\right)
		\end{align}
		for the parameters.
	Note that if $v,w\in \C^n$ satisfy the condition $\operatorname{Im}(v\cdot w) = 0 =\operatorname{Im}(\bar{v}\cdot w)$, then we must have \begin{equation}\label{eq.orthogonality}\operatorname{Re}(v) \perp \operatorname{Im}(w), \operatorname{Im}(v) \perp \operatorname{Re}(w)\,.\end{equation}
	Since we are still free to choose a coordinate system in $\R^4$ we can assume without loss of generality that $c=e_1+ i e_2$ and so $|c|^2=2$. Writing $b=\beta + i \gamma$ for $\beta=(\beta_1,\beta_2,\beta'), \gamma=(\gamma_1,\gamma_2,\gamma') \in \R^4$ . 
	Hence we have \[(\bar{c}+ p^2c)=(1+p^2)e_1 - i(1-p^2)e_2 \text{ and } (b+\bar{b}p^2)=(1+p^2)\beta + i (1-p^2)\gamma\]
	we deduce from \eqref{eq.orthogonality} applied to the combination \eqref{eq.condition1}\&\eqref{eq.condition3} and the combination \eqref{eq.condition2}\&\eqref{eq.condition4} that
	\[ \gamma_1=\gamma \cdot e_1 =0 =\beta \cdot e_2=\beta_2\,.\]
	But this implies that $b^2$ translates to 
	\begin{equation}\label{eq.b^2=0}
		\beta'\cdot \gamma' =0 \text{ and } |\beta'|^2=\gamma_2^2, |\gamma'|^2=\beta_1^2 \Rightarrow |b|^2=2( \beta_1^2+\gamma_2^2) 
	\end{equation}
	Hence we can calculate \eqref{eq.condition1}-\eqref{eq.condition4} more explicitly to be 
	\begin{align}
			\label{eq.condition1-1}0&=(1+p^2) \beta_1 - (1-p^2) \gamma_2 + 2 p^2 \\
			\label{eq.condition2-1}0&=\frac{p^2}{(1+p^2)^2} \left((1+p^2)\beta_1 - (1-p^2) \gamma_2 \right) + |b|^2 \\
			\label{eq.condition3-1}0&=(1+p^2) \beta_1 + (1-p^2) \gamma_2 + |b|^2\\
			\label{eq.condition4-1}0&=(1+p^2)\beta_1 + (1-p^2) \gamma_2 + \frac{2p^4}{(1+p^2)^2}	\end{align}
	Comparing the last two immediately gives
	\[|b|^2 = \frac{2p^4}{(1+p^2)^2}\]
	(This is consistent with the comparison between \eqref{eq.condition1-1}\&\eqref{eq.condition2-1}.)
	Adding \eqref{eq.condition1-1} to \eqref{eq.condition3-1} and subtracting them gives
	\begin{align*}-2(1+p^2)\beta_1 = |b|^2 + 2p^2=2p^2\, \frac{(1+p^2)^2 +p^2}{(1+p^2)^2} \\
	2(1-p^2) \gamma_2 = 2p^2 -|b|^2=2p^2\, \frac{(1+p^2)^2-p^2}{(1+p^2)^2} \end{align*}
	The second equality forces $0< p<1$, and we summarize  
	\begin{align*}
		\beta_1=- p^2\, \frac{(1+p^2)^2 +p^2}{(1+p^2)^3} \text{ and } \gamma_2 = p^2\, \frac{(1+p^2)^2-p^2}{(1+p^2)^2(1-p^2)}
	\end{align*}
	Finally we may recall \eqref{eq.b^2=0} hence $0<p<1$ must be a zero of 
	\[2\beta_1(p)^2+2\gamma_2(p)^2 = \frac{2p^4}{(1+p^2)^2} \]
	or 
	\[\frac{ \left((1+p^2)^2+p^2\right)^2}{(1+p^2)^4}+ \frac{\left((1+p^2)^2-p^2\right)^2}{(1+p^2)^2(1-p^2)^2}=1\,.\]
	But since the function 
	\[f(x) = \frac{1}{(1+x)^2} \left(\frac{ \left((1+x)^2+x\right)^2}{(1+x)^2}+ \frac{\left((1+x)^2-x\right)^2}{(1-x)^2} \right)\]
	is increasing on the interval $(0,1)$ and satisfies $f(0)=2$, there is no solution $x=p^2\in(0,1)$ for $f(x)=1$. We conclude that there is no $\R P^2$ with two ends. 
\end{proof}

\section{Lower semi-continuity of the index} \label{appendix:lowersemicontinuity}
We split the version of lower semi-continuity needed for our argument into two steps. The first (Proposition~\ref{prop.lower semicontinuity1}) addresses lower semi-continuity of the Willmore Morse index for a sequence of closed, compact Willmore immersions that converge in $C^3$. 
The second (Proposition~\ref{prop.lower semicontinuity2}) shows that the inverted and compactified Willmore immersions associated to a convergent sequence of complete minimal surfaces (such as a sequence arising from the $SO(n,\C)$-deformations in Lemma~\ref{lem:genericS}) must actually converge in $C^3$.

\begin{proposition}\label{prop.lower semicontinuity1}
If $f_k \colon \Sigma \to \R^n$ is a sequence of closed Willmore immersions converging in $C^3$ to a closed Willmore immersion $f \colon \Sigma \to \R^n,$ then
	\begin{equation}\label{eq.lowersemicontinuity}
		\Ind(f)\le \lim_k \Ind(f_k)\,.
	\end{equation}
\end{proposition}

\begin{proof}%[Proof of Proposition \ref{prop.lower semicontinuity1}]
Recall that the index of a bilinear operator such as $\delta^2\mathcal{W}(f)$ is characterized as
\[ \Ind(f)=%\operatorname{Ind}(\delta^2 \mathcal{W}(f))=
 \max \{ \dim(L)\colon L  \text{  a linear subspace of } \Gamma(N_{f}\Sigma) \text{ s.t.}\ \delta^2\mathcal{W}(f)|_L<0\}. \]
 Hence \eqref{eq.lowersemicontinuity} follows by showing that for every $\vec{v} \in \Gamma(N_{f}\Sigma)$ there exists a sequence $\vec{v}_k \in \Gamma(N_{f_k}\Sigma)$ with 
 \begin{equation}\label{eq.recoverysequence}
 	\lim_k \delta^2\mathcal{W}(f_k)(\vec{v}_k,\vec{v}_k) \le \delta^2\mathcal{W}(f)(\vec{v},\vec{v})\,.
 \end{equation}
Since $f_k \to f$ in $C^3$ we can define projection operators $P_k(x)$ that are the orthogonal projections of $N_{f_k(x)}\Sigma$ onto $N_{f(x)}\Sigma$. They converge under the above assumptions uniformly in $C^2(\Sigma)$. 
Given $\vec{v} \in \Gamma(N_f \Sigma)$ we just choose $\vec{v}_k(x)=P_k(x)\vec{v}$ and note that $\vec{v}_k \to \vec{v}$ in $C^2(\Sigma)$. Because the $\mathcal{W}$-Jacobi operator $Z_f$ is a fourth-order operator whose coefficients  depend\footnote{The first variation depends only on $\vec H$ and second order derivatives of $\vec H$, hence the second variation depends on at most three derivatives of $A$. After integrating by parts this implies at most one derivative on $A$ for the coefficients.} only on third derivatives of $f$, we deduce \eqref{eq.recoverysequence}.
\end{proof}

\begin{definition}
 A sequence of meromorphic differentials $\phi_k$ converges to $\phi$ if for every complex chart $z\colon U \to B_1 \subset \C$ and  $\phi_k(z)= m_k(z)\, dz, \phi(z)=m(z)\, dz$ the meromorphic functions 
 satisfy $\lim_k m_k(z)=m(z)$ in the topology of the Riemann sphere $\Chat$.
\end{definition}

\begin{proposition}\label{prop.lower semicontinuity2}
	Let $X_k\colon\Sigma \setminus\{p_1^k,...,p_m^k\}  \to \R^n$ be a sequence of complete minimal surfaces with planar embedded ends at $\{p_1^k, \dotsc, p_m^k\} \subset \Sigma$ that induce the same conformal structure $J \colon \Sigma \to \Sigma$. 
	 Suppose the meromorphic differentials $\partial_z X_k \, dz$ converge to the meromorphic differential $\partial_z X \, dz$ of a complete minimal surface $X\colon\Sigma \setminus\{p_1,...,p_{\tilde m}\}  \to \R^n$ with planar embedded ends at $\{p_1, \dots, p_{\tilde{m}}\}$, and $0 \notin X(\Sigma)$. If there is one point $q \in \Sigma\setminus \{p_1, \dotsc, p_{\tilde{m}}\}$ such that $\lim_k X_k(q) = X(q)$ then the inverted and compactified Willmore surfaces converge in $C^3$:
	\[f_k = \frac{X_k}{|X_k|^2} \overset{C^3}{\to} f=\frac{X}{|X|^2}\]	
\end{proposition}

\begin{proof}%[Proof of Proposition \ref{prop.lower semicontinuity2}]
To deduce the convergence of $f_k$ one could use the regularity theory for Willmore surfaces. But since there is a direct proof using the Weierstrass representation, we present it here. \\

Let $z \colon U \to B_1 \subset \C$ be a conformal chart with $U \cap \{p_1, \dotsc, p_{\tilde{m}}\} = \emptyset$. This implies that (in these coordinates) $\partial_z X(z)$ is holomorphic. Thus, the convergence 
$\partial_z X_k(z) \, dz \to \partial_z X(z)\,dz$
implies that $\lim_k \partial_z X_k(z) = \partial_z X(z)$ pointwise in $B_{\frac12}$. Since this is a sequence of holomorphic functions, we deduce that $\partial_z X_k \xrightarrow{C^\infty} \partial_z X$.
The fact that $\Sigma \setminus \{p_1, \dotsc, p_{\tilde{m}}\}$ is connected and that there is one point $q$ in this set where $\lim_k X_k(q) = X(q)$ the above consideration implies that $\lim_k X_k = X$ locally uniformly in $C^l(\Sigma \setminus \{p_1, \dotsc, p_{\tilde{m}}\})$ for each $l$. Because $0 \notin X(\Sigma)$, there is  a constant $c>0$ such that $|X(p)|\ge 2c$ for all $p \in \Sigma$.
Given any compact set $K \Subset \Sigma \setminus \{p_1, \dotsc, p_{\tilde{m}}\}$ we have $|X_k(p)|>c$ for all $p \in K$ and $k$ sufficiently large
due to the local strong convergence. Thus we conclude the local smooth convergence of $f_k$ to $f$ on $\Sigma \setminus \{p_1, \dotsc, p_{\tilde{m}}\}$.

It remains to check the convergence in a neighborhood of each end $p_j$: We fix a complex coordinate $z$ around $p_j$ such that $z(p_j)=0$. Hence without loss of generality we have in $B_{2\epsilon}$ the expansion
\[ X_z\,dz = \left(-\frac{a}{z^2} + Y(z)\right)\, dz  \]
with $|a|^2=2, a^2=0$ and $Y$ being a holomorphic bounded function. From the convergence of meromorphic differentials, we deduce that $X_k$ has a similar expansion:
\[ (X_k)_z\,dz= \left(- \frac{a_k}{(z-z_k)^2} + Y_k(z)\right) \, dz \]
where $a_k^2 =0, |a_k|^2 \to 2$ as $|z_k|\downarrow 0$, and with $Y_k$ holomorphic and bounded converging to $Y$.
These expansions imply that
\begin{align*}
	X(z)&= \Re\left(\frac{a}{z} + \int_0^z Y(w)\,dw +c\right)=\Re\left( \frac{a}{z}+ \hat{Y}(z)\right)\\
	X_k(z)&=\Re\left(\frac{a_k}{z-z_k} + \int_0^z Y_k(w)\,dw +c_k\right)=\Re\left( \frac{a_k}{z-z_k}+ \hat{Y}_k(z)\right)\,.
\end{align*}
Due to the above-established convergence outside the ends we have that the constants satisfy $c_k \to c$. Using $2|\Re(q)|^2=|q|^2 + \Re(q^2)$
we compute 
\begin{align*}
 2 |z-z_k|^2  |X_k(z)|^2 &= |a_k + (z-z_k)\hat{Y}_k(z)|^2 \\
& \qquad\qquad + \Re\left( 2\overline{(z-z_k)} \,a_k\cdot \hat{Y}_k(z) + |z-z_k|^2\hat{Y}_k(z)^2\right)\,,
\end{align*}
where we have made essential use of $a_k^2=0$. The right hand side is clearly bounded and converges in $C^\infty$ to
\[2|z|^2|X(z)|^2 = |a + z\hat{Y}(z)|^2 + \Re\left( 2\overline{z} \,a\cdot \hat{Y}(z) + |z|^2\hat{Y}(z)^2\right)\,.\]
Since the same holds true for 
\begin{align*}
	2 |z-z_k|^2X_k(z) = &2\Re\left( ( \overline{z-z_k}) a_k + |z-z_k|^2 \hat{Y}_k(z)\right) \to \\
	& \qquad\qquad\qquad 2\Re\left( \overline{z} a + |z|^2 \hat{Y}(z)\right)=2|z|^2X(z)\,.
\end{align*}
The smooth convergence of $f_k(z)$ to $f(z)$ follows now also in a neighborhood of each end.
\end{proof}

\section{Self-adjointness of the Willmore-Jacobi operator}\label{appendix:selfadjoint}
In this section we argue why the index form associated to the Willmore energy on a compact surface is symmetric. We also give a short argument why the $\mathcal{W}$-Jacobi operator is $L^2$-self-adjoint. 

\smallskip
\noindent \emph{Step 1: } The index form corresponding to the second variation of $\W$ is a symmetric bilinear form on the space of $C^2$-sections of the normal bundle on a compact surface. %symmetry and order of the operator) 
\smallskip

\emph{Proof of step 1: } First we recall that the Willmore energy is a second order geometric energy on the space of immersions  that only depends on first and second derivatives of the immersion $f:\Sigma\to\mathbb{R}^n$, i.e. 
\[\mathcal{W}(f)= \frac14 \int_\Sigma H(D f, D^2 f) d\mu\]
where $H(P,M)$ is an analytic map. 
For any two $C^2$-regular vectorfields $\vec{v}_i \in \Gamma(N_f \Sigma)$ we define the function 
\[\mathcal{F}(x_1,x_2):= \,\,\mathcal{W}(f + x^1 \vec{v}_1 + x^2 \vec{v}_2)\,\]
for $(x_1,x_2)\in (-\epsilon, \epsilon)^2$. 
The analyticity of $\mathcal{F}$ in $(x_1,x_2)$ implies that its Hessian is a symmetric linear map at $(0,0)$.
% 
% The function $f$ is analytic in $x_i$ and thus we deduce that the Hessian is symmetric linear map. 
Since $\Sigma$ is compact we deduce that $\delta^2\mathcal{W}(f)$ is a bilinear symmetric form on $C^2$-regular sections $\vec{v}\in \Gamma(N_f \Sigma)$, i.e. 
\begin{equation}\label{eq.1expansion}
	\delta^2\mathcal{W}(f)(\vec{v},\vec{w}) = \int_{\Sigma} \sum_{|\alpha|,|\beta|\le 2} A_{\alpha,\beta} \partial^\alpha\vec{v}\partial^\beta\vec{w} \,d\mu\,,
\end{equation}
where the coefficient matrices are algebraic expressions in at most two derivatives of the immersion $f$. The symmetry assures that $A_{\alpha,\beta}=A_{\beta,\alpha}$.
Hence we get that the associated $\mathcal{W}$-Jacobi operator $Z$ is formally self-adjoint (also called \emph{symmetric}) on $C^2$-regular sections. 

\medskip
\noindent \emph{Step 2: } The closure of $Z$ on $L^2$ is self-adjoint, and the eigenvalues of $Z$ consist of a (discrete) sequence of numbers (with finite multiplicity) which is bounded from below and with $\infty$ as its only accumulation point.  
\smallskip

\emph{Proof of Step 2: } In order to study the eigenvalues of $Z$ and to show that the closure of $Z$ is $L^2$-self-adjoint, we are interested in the highest-order part of $Z$. We use the specific structure of the Willmore integrand $H(Df,D^2f)= |\vec{H}|^2\,d\mu$. We only need to determine $|\partial_{t} \vec{H}|^2$ for the family $f_t= f + t\vec{v}$. It is classical that $\partial_{t}\vec{H}= \Delta^\perp \vec{v} - F(A^2)\vec{v}$, where $F(A^2)$ denotes a quadratic operator in the second fundamental form $A$ \cite[p.\ 10]{KuwertSch}. 
We deduce that the highest order part in \eqref{eq.1expansion}
is $|\Delta^\perp \vec{v}|^2$. Thus,  a coercivity estimate follows from (\ref{eq.1expansion}): for all $\vec{v} \in \Gamma(N_f \Sigma)$ one has 
\begin{equation}\label{eq.2expansion}
	\delta^2\mathcal{W}(f)(\vec{v},\vec{v}) \ge \frac{1}{4} \int_{\Sigma} |\Delta^\perp \vec{v}|^2 \, d\mu\,  - C \int_{\Sigma} |D^3\vec{v}|^2 +|D^2\vec{v}|^2 +|D\vec{v}|^2 + |\vec{v}|^2 \, d\mu\,.
\end{equation}
Combining \eqref{eq.1expansion} \& \eqref{eq.2expansion}, and using interpolation, we have for all $\vec{v} \in \Gamma(N_f \Sigma)$
\[ \frac18 \int_{\Sigma} |D^2\vec{v}|^2\, d\mu - C \int_{\Sigma} |\vec{v}|^2 \, d\mu  \le \delta^2 \mathcal{W}(f)(\vec{v},\vec{v}) \le C \int_{\Sigma} |D^2\vec{v}|^2 + |\vec{v}|^2\, d\mu.\]
Applying classical functional analysis to the operator $Z:W^{2,2} \to L^2\subset \left(W^{2,2}\right)'$ we deduce self-adjointness \cite[Satz~VII.2.11 on p.~378]{Werner} and the discreteness of the spectrum \cite[Section~6.5.1]{Evans}.

\bibliographystyle{plain}
\bibliography{Lit_2}

\begin{thebibliography}{10}

\bibitem{Apery}
Fran\c{c}ois Ap\'{e}ry.
\newblock An algebraic halfway model for the eversion of the sphere.
\newblock {\em Tohoku Math. J. (2)}, 44(1):103--150, 1992.
\newblock With an appendix by Bernard Morin.

\bibitem{Banchoff}
Thomas~F. Banchoff.
\newblock Triple points and surgery of immersed surfaces.
\newblock {\em Proc. Amer. Math. Soc.}, 46:407--413, 1974.

\bibitem{Bauer}
Matthias {Bauer} and Ernst {Kuwert}.
\newblock {Existence of minimizing Willmore surfaces of prescribed genus}.
\newblock {\em {Int. Math. Res. Not.}}, 2003(10):553--576, 2003.

\bibitem{Riviere2013}
Yann Bernard and Tristan Rivi\`ere.
\newblock Singularity removability at branch points for {W}illmore surfaces.
\newblock {\em Pacific J. Math.}, 265(2):257--311, 2013.

\bibitem{Blaschke}
Wilhelm {Blaschke}.
\newblock {Vorlesungen \"uber Differentialgeometrie und geometrische Grundlagen
  von Einsteins Relativit\"atstheorie. III: Differentialgeometrie der Kreise
  und Kugeln. Bearbeitet von {\textit G. Thomsen}.}
\newblock {X + 474 S. Berlin, J. Springer (Grundlehren der mathematischen
  Wissenschaften in Einzeldarstellungen)}, 1929.

\bibitem{BryantDuality}
Robert~L. Bryant.
\newblock A duality theorem for {W}illmore surfaces.
\newblock {\em J. Differential Geom.}, 20(1):23--53, 1984.

\bibitem{Bryant}
Robert~L. Bryant.
\newblock Surfaces in conformal geometry.
\newblock In {\em The mathematical heritage of {H}ermann {W}eyl ({D}urham,
  {NC}, 1987)}, volume~48 of {\em Proc. Sympos. Pure Math.}, pages 227--240.
  Amer. Math. Soc., Providence, RI, 1988.

\bibitem{CastroUrbano}
Ildefonso Castro and Francisco Urbano.
\newblock Lagrangian surfaces in the complex {E}uclidean plane with conformal
  {M}aslov form.
\newblock {\em Tohoku Math. J. (2)}, 45(4):565--582, 1993.

\bibitem{Colding}
Tobias~Holck Colding and William~P. Minicozzi, II.
\newblock {\em A course in minimal surfaces}, volume 121 of {\em Graduate
  Studies in Mathematics}.
\newblock American Mathematical Society, Providence, RI, 2011.

\bibitem{EkholmWhiteWienholtz}
Tobias Ekholm, Brian White, and Daniel Wienholtz.
\newblock Embeddedness of minimal surfaces with total boundary curvature at
  most {$4\pi$}.
\newblock {\em Ann. of Math. (2)}, 155(1):209--234, 2002.

\bibitem{Evans}
Lawrence~C. Evans.
\newblock {\em Partial differential equations}, volume~19 of {\em Graduate
  Studies in Mathematics}.
\newblock American Mathematical Society, Providence, RI, second edition, 2010.

\bibitem{Francis97}
George Francis, John~M. Sullivan, and Chris Hartman.
\newblock Computing sphere eversions.
\newblock In {\em Mathematical visualization ({B}erlin, 1997)}, pages 237--255.
  Springer, Berlin, 1998.

\bibitem{Francis95}
George Francis, John~M. Sullivan, Rob~B. Kusner, Ken~A. Brakke, Chris Hartman,
  and Glenn Chappell.
\newblock The minimax sphere eversion.
\newblock In {\em Visualization and mathematics ({B}erlin-{D}ahlem, 1995)},
  pages 3--20. Springer, Berlin, 1997.

\bibitem{FrancisMorin}
George~K. Francis and Bernard Morin.
\newblock Arnold {S}hapiro's eversion of the sphere.
\newblock {\em Math. Intelligencer}, 2(4):200--203, 1979/80.

\bibitem{Heller}
Sebastian Heller.
\newblock Willmore spheres in the {$3$}-sphere revisited.
\newblock arXiv:2003.06922, to appear in Comm. Anal. Geo.

\bibitem{HirschEMB}
Jonas Hirsch and Elena M\"{a}der-Baumdicker.
\newblock On the index of {W}illmore spheres.
\newblock {\em J. Differential Geom.}, 124(1):37--79, 2023.

\bibitem{HoffmanMeeks}
David Hoffman and William~H. Meeks, III.
\newblock Embedded minimal surfaces of finite topology.
\newblock {\em Ann. of Math. (2)}, 131(1):1--34, 1990.

\bibitem{HoffmanOsserman}
David~A. Hoffman and Robert Osserman.
\newblock The geometry of the generalized {G}auss map.
\newblock {\em Mem. Amer. Math. Soc.}, 28(236):iii+105, 1980.

\bibitem{Kusner}
Rob Kusner.
\newblock Conformal geometry and complete minimal surfaces.
\newblock {\em Bull. Amer. Math. Soc. (N.S.)}, 17(2):291--295, 1987.

\bibitem{KusnerII}
Rob Kusner.
\newblock Comparison surfaces for the {W}illmore problem.
\newblock {\em Pacific J. Math.}, 138(2):317--345, 1989.

\bibitem{Kusner94}
Rob Kusner.
\newblock Estimates for the biharmonic energy on unbounded planar domains, and
  the existence of surfaces of every genus that minimize the
  squared-mean-curvature integral.
\newblock In {\em Elliptic and parabolic methods in geometry ({M}inneapolis,
  {MN}, 1994)}, pages 67--72. A K Peters, Wellesley, MA, 1996.

\bibitem{Kusner16}
Rob Kusner, Andrea Mondino, and Felix Schulze.
\newblock Willmore bending energy on the space of surfaces, {MSRI}
  \emph{Emissary}, Spring 2016.

\bibitem{KusnerSchmitt}
Rob Kusner and Nick Schmitt.
\newblock The spinor representations of surfaces in space.
\newblock arXiv:dg-ga/9610005v1, 1996.

\bibitem{KusnerThesis88}
Robert~Barnard Kusner.
\newblock {\em Global geometry of extremal surfaces in three space}.
\newblock ProQuest LLC, Ann Arbor, MI, 1988.
\newblock Thesis (Ph.D.)--University of California, Berkeley.

\bibitem{Kuwert2004}
Ernst Kuwert and Reiner Sch\"{a}tzle.
\newblock Removability of point singularities of {W}illmore surfaces.
\newblock {\em Ann. of Math. (2)}, 160(1):315--357, 2004.

\bibitem{Kuwert2008}
Ernst Kuwert and Reiner Sch\"{a}tzle.
\newblock Branch points of {W}illmore surfaces.
\newblock {\em Duke Math. J.}, 138(2):179--201, 2007.

\bibitem{KuwertSch}
Ernst Kuwert and Reiner Sch\"{a}tzle.
\newblock The {W}illmore functional.
\newblock In {\em Topics in modern regularity theory}, volume~13 of {\em CRM
  Series}, pages 1--115. Ed. Norm., Pisa, 2012.

\bibitem{Lamm1}
Tobias Lamm, Jan Metzger, and Felix Schulze.
\newblock Foliations of asymptotically flat manifolds by surfaces of {W}illmore
  type.
\newblock {\em Math. Ann.}, 350(1):1--78, 2011.

\bibitem{Lawson}
H.~Blaine Lawson, Jr.
\newblock Complete minimal surfaces in {$S^{3}$}.
\newblock {\em Ann. of Math. (2)}, 92:335--374, 1970.

\bibitem{Lee2}
Hojoo Lee.
\newblock Minimal surfaces in {$\mathbb{R}^4$} foliated by conic sections and
  parabolic rotations of holomorphic null curves in {$\mathbb{ C}^4$}.
\newblock {\em J. Korean Math. Soc.}, 57(1):1--19, 2020.

\bibitem{Lee}
Jaehoon Lee.
\newblock Minimal surfaces in {$\mathbb{R}^4$} like the {L}agrangian catenoid.
\newblock {\em J. Geom. Anal.}, 32(3):Paper No. 98, 22, 2022.

\bibitem{LiYau}
Peter Li and Shing~Tung Yau.
\newblock A new conformal invariant and its applications to the {W}illmore
  conjecture and the first eigenvalue of compact surfaces.
\newblock {\em Invent. Math.}, 69(2):269--291, 1982.

\bibitem{MarquesNevesWillmore}
Fernando~C. Marques and Andr{\'e} Neves.
\newblock Min-max theory and the {W}illmore conjecture.
\newblock {\em Ann. of Math. (2)}, 179(2):683--782, 2014.

\bibitem{MarquesNeves}
Fernando~C. Marques and Andr{\'e} Neves.
\newblock The {W}illmore conjecture.
\newblock {\em Jahresber. Dtsch. Math.-Ver.}, 116(4):201--222, 2014.

\bibitem{BanchoffMax}
Nelson Max and Tom Banchoff.
\newblock Every sphere eversion has a quadruple point.
\newblock In {\em Contributions to analysis and geometry ({B}altimore, {M}d.,
  1980)}, pages 191--209. Johns Hopkins Univ. Press, Baltimore, Md., 1981.

\bibitem{Michelat}
Alexis Michelat.
\newblock On the {M}orse index of {W}illmore spheres in {$S^3$}.
\newblock {\em Comm. Anal. Geom.}, 28(6):1337--1406, 2020.

\bibitem{MichelatBranched}
Alexis Michelat.
\newblock On the {M}orse index of branched {W}illmore spheres in 3-space.
\newblock {\em Calc. Var. Partial Differential Equations}, 60(4):Paper No. 126,
  97, 2021.

\bibitem{Montiel2}
Sebasti\'{a}n Montiel.
\newblock Willmore two-spheres in the four-sphere.
\newblock {\em Trans. Amer. Math. Soc.}, 352(10):4469--4486, 2000.

\bibitem{MontielRos}
Sebasti\'{a}n Montiel and Antonio Ros.
\newblock Schr\"{o}dinger operators associated to a holomorphic map.
\newblock In {\em Global differential geometry and global analysis ({B}erlin,
  1990)}, volume 1481 of {\em Lecture Notes in Math.}, pages 147--174.
  Springer, Berlin, 1991.

\bibitem{Morin78c}
Bernard Morin.
\newblock \'{E}quations du retournement de la sph\`ere.
\newblock {\em C. R. Acad. Sci. Paris S\'{e}r. A-B}, 287(13):A879--A882, 1978.

\bibitem{Morin78}
Bernard Morin and Jean-Pierre Petit.
\newblock Le retournement de la sphere.
\newblock {\em C. R. Acad. Sci. Paris Ser. A-B}, 287(11):A791--A794, 1978.

\bibitem{Morin78b}
Bernard Morin and Jean-Pierre Petit.
\newblock Probl\'{e}matique du retournement de la sph\`ere.
\newblock {\em C. R. Acad. Sci. Paris S\'{e}r. A-B}, 287(10):A767--A770, 1978.

\bibitem{Peng}
Chia-Kuei Peng and Liang Xiao.
\newblock Willmore surfaces and minimal surfaces with flat ends.
\newblock In {\em Geometry and topology of submanifolds, {X}
  ({B}eijing/{B}erlin, 1999)}, pages 259--265. World Sci. Publ., River Edge,
  NJ, 2000.

\bibitem{PerezRos}
Joaqu\'{i}n P\'{e}rez and Antonio Ros.
\newblock The space of properly embedded minimal surfaces with finite total
  curvature.
\newblock {\em Indiana Univ. Math. J.}, 45(1):177--204, 1996.

\bibitem{Riviere2008}
Tristan Rivi\`ere.
\newblock Analysis aspects of {W}illmore surfaces.
\newblock {\em Invent. Math.}, 174(1):1--45, 2008.

\bibitem{Riviere2018}
Tristan Rivi\`ere.
\newblock Willmore minmax surfaces and the cost of the sphere eversion.
\newblock {\em J. Eur. Math. Soc. (JEMS)}, 23(2):349--423, 2021.

\bibitem{RosenbergToubiana}
Harold Rosenberg and \'{E}ric Toubiana.
\newblock Some remarks on deformations of minimal surfaces.
\newblock {\em Trans. Amer. Math. Soc.}, 295(2):491--499, 1986.

\bibitem{Schoen}
Richard {Schoen}.
\newblock {Uniqueness, symmetry, and embeddedness of minimal surfaces.}
\newblock {\em {J. Differ. Geom.}}, 18:791--809, 1983.

\bibitem{Smale}
Stephen Smale.
\newblock A classification of immersions of the two-sphere.
\newblock {\em Trans. Amer. Math. Soc.}, 90:281--290, 1958.

\bibitem{Thomsen}
G.~Thomsen.
\newblock Grundlagen der konformen {F}l\"achentheorie.
\newblock {\em Abh. Math. Sem. Univ. Hamburg}, 3(1):31--56, 1924.

\bibitem{Werner}
Dirk Werner.
\newblock {\em Funktionalanalysis}.
\newblock Springer-Lehrb. Berlin: Springer Spektrum, 8th revised edition
  edition, 2018.

\end{thebibliography}

\end{document}